\documentclass[12pt, reqno]{amsart}
\usepackage{amsmath, amsthm, amscd, amsfonts, amssymb, graphicx, color,tikz}
\usepackage[bookmarksnumbered, colorlinks, plainpages]{hyperref}
\usepackage{mathrsfs}
\usepackage{enumerate}
\usepackage[utf8]{inputenc}

\newtheorem{theorem}{Theorem}[section]
\newtheorem{lemma}[theorem]{Lemma}
\newtheorem{proposition}[theorem]{Proposition}
\newtheorem{corollary}[theorem]{Corollary}
\theoremstyle{definition}

\newtheorem{example}[theorem]{Example}

\newtheorem{remark}[theorem]{Remark}
\newtheorem{question}[theorem]{Question}
\newtheorem*{theoremllsr}{Theorem A}
\numberwithin{equation}{section}

\newcommand{\A}{\mathrm{Re\,}}
\newcommand{\B}{\mathrm{Im\,}}

\newcommand{\veps}{\varepsilon}
\newcommand{\CC}{\mathbb C}
\newcommand{\RR}{\mathbb R}
\newcommand{\DD}{\mathbb D}
\newcommand{\TT}{\mathbb T}
\newcommand{\ZZ}{\mathbb Z}
\newcommand{\NN}{\mathbb N}

\newcommand{\cha}{\mathrm{char\,}}
\newcommand{\alg}{\mathrm{alg\,}}

\newcommand{\ran}{\mathrm{Ran\,}}
\newcommand{\GD}{\mathcal{GD}_{>0}}
\newcommand{\GH}{\mathcal{G}_{>0}}
\newcommand{\GG}{\mathcal{GH}_{>0}}

\newcommand{\hinfplus}{\mathcal H^\infty_+}

\allowdisplaybreaks[4]

\hypersetup{colorlinks=true,linkcolor=blue,citecolor=green}

\begin{document}
\setcounter{page}{1}

\title[K\"{o}nigs maps and commutants of composition operators]
{K\"{o}nigs maps and commutants of composition operators on the Hardy-Hilbert space of Dirichlet series}
\date{\today}

\author{Fr\'ed\'eric Bayart}
\address{Laboratoire de Math\'ematiques Blaise Pascal UMR 6620 CNRS, Universit\'e Clermont Auvergne, Campus universitaire C\'ezeaux, 3 place Vasarely, 63178 Aubi\`ere Cedex, France.}
\email{frederic.bayart@uca.fr}

\author{Xingxing Yao}
\address{School of Mathematics and Physics, Wuhan Institute of Technology, Wuhan 430205, China.}\email{xxyao.math@wit.edu.cn}




\begin{abstract}
Let $\varphi$ be a holomorphic map which is a symbol of a bounded composition operator $C_\varphi$ acting on the Hardy-Hilbert space of Dirichlet series.
We find a K\"{o}nigs map for $\varphi$.
We then deduce several applications on $C_\varphi$ (e.g. on its spectrum, on its dynamical properties).
In particular, we study for a large class of symbols $\varphi$ if the associated composition operator has a minimal commutant.
\end{abstract}
\maketitle

\tableofcontents

\section{Introduction}

\subsection{Composition operators on the Hardy-Hilbert space of Dirichlet series}
From the pioneering work of Gordon and Hedenmalm \cite{gh}, the study of composition
operators on the Hardy-Hilbert space of Dirichlet series has been the subject of many investigations.
Let us introduce some terminology. For $\theta\in\mathbb R$, $\mathbb{C}_{\theta}$ will denote the half-plane $\{s\in\mathbb{C}: \A s>\theta\}$.
We denote by $\mathcal D$ the set of holomorphic functions $f$ defined on some open set $U\subset\mathbb C$
where $U$ contains a half-plane $\mathbb C_\theta$ and such that $f$ is equal to the sum of a convergent Dirichlet series $\sum_{n\geq 1}a_n n^{-s}$
in some (possibly remote) half-plane $\mathbb C_\nu$.

Hedenmalm-Lindqvist-Seip \cite{hls} introduced and started the study of the Hilbert space of Dirichlet series with square summable coefficients:
\begin{equation*}
\mathcal{H}^{2}=\bigg\{f(s)=\sum_{n=1}^{\infty}a_{n}n^{-s}:\|f\|=\bigg(\sum_{n=1}^{\infty}|a_{n}|^{2}\bigg)^{1/2}<\infty\bigg\}.
\end{equation*}
By the Cauchy-Schwarz inequality, the functions in $\mathcal{H}^{2}$ are all
analytic on the half-plane $\mathbb{C}_{1/2}$.
The main results in \cite{gh} (see also \cite{qs}) show that
an analytic map $\varphi:\mathbb{C}_{1/2}\rightarrow\mathbb{C}_{1/2}$ induces
a bounded composition operator $C_{\varphi}:f\mapsto f\circ\varphi$ on $\mathcal{H}^{2}$ if and only if it is a member of the Gordon-Hedenmalm class $\mathcal{G}$ defined as follows. \medskip

\noindent \textbf{Definition.}
The \textbf{Gordon-Hedenmalm class} $\mathcal{G}$ consists of the maps $\varphi:\mathbb{C}_{1/2}\rightarrow\mathbb{C}_{1/2}$ of the form
$$\varphi(s)=c_{0}s+\psi(s),$$
where $c_{0}$ is a non-negative integer (called the \textbf{characteristic} of $\varphi$, i.e., $\cha(\varphi)=c_{0}$),
$\psi\in\mathcal D$ converges uniformly in $\CC_\veps$ for every $\veps>0$ and
\begin{itemize}
 \item[(a)] $\psi(\CC_0)\subset\CC_{1/2}$ if $c_0=0$;
 \item[(b)] $\psi\equiv i\tau$ for some $\tau\in\mathbb R$ or $\psi(\CC_0)\subset \CC_0$ if $c_0\geq 1$.
\end{itemize}
We shall denote by $\GH$ the class of symbols $\varphi\in\mathcal G$ with positive characteristic.
\medskip

From then on, several authors have studied the properties of composition operators acting on $\mathcal{H}^{2}$
or on similar spaces of Dirichlet series, see for instance \cite{bai,baib,b1,b2,b2024,bb,bwy,bp,bp20,fi,qs}.
In this paper, we shall contribute to this circle of ideas, for composition operators whose symbol has positive characteristic.
This will be done through the look for a K\"{o}nigs map for $\varphi$ and through the consequences that we will get from it.

\subsection{K\"{o}nigs maps, Schr\"{o}der's equation and consequences}

Let us come back for a while to the unit disc $\mathbb D$ and let us consider $\varphi:\DD\to\DD$ holomorphic
and its associated composition operator $C_\varphi:f\mapsto f\circ\varphi$ which is bounded on $H(\DD)$
as well as on $H^2(\DD)$. The eigenfunction equation for $C_\varphi$ is called Schr\"{o}der's equation:
$$f\circ\varphi=\lambda f,\ \lambda\in\CC,\ f\in H(\DD).$$
If $\varphi$ fixes the origin, satisfies $\varphi'(0)\neq 0$ and is not an elliptic automorphism, this equation is linked to the existence of a K\"{o}nigs map for $\varphi,$
namely a holomorphic map $u:\DD\to\CC$ such that
$$u\circ\varphi=\varphi'(0) u.$$
If $\varphi$ has Denjoy-Wolff attractive fixed point on the unit circle $\TT,$
with $\varphi'(\xi)\in (0,1]$, a K\"{o}nigs map of $\varphi$ is now a map $u:\DD\to \CC$ such that
$$
\begin{array}{ll}
 u\circ\varphi=\varphi'(\xi)u&\textrm{if }\varphi'(\xi)\in(0,1)\\
 u\circ\varphi=u+a&\textrm{if }\varphi'(\xi)=1\textrm{ for some }a\in\CC.
\end{array}
$$
The existence of these maps has been discussed in details in \cite{BP79,Pom79}.

It turns out that solving Schr\"{o}der's equation or finding a K\"{o}nigs map have important consequences for the study of composition operators,
e.g.\ for compactness or cyclicity questions (cf.\ \cite{bs}),
for computing the spectrum on various spaces of holomorphic functions (cf.\ \cite{ABCC,cm})
and for studying the commutant of a composition operator (cf.\ \cite{llsr18}).

Our first aim, in this paper, will be to make these tools at our disposal to study composition operators
on $\mathcal H^2$. Since a symbol $\varphi\in\GH$ is a self-map of $\CC_0$, we can use the results of \cite{BP79,Pom79}
for the existence of $u$, but since we are working in the context of Dirichlet series, we need slightly more.
We do expect that $u$ itself has the shape of a symbol
of a composition operator on $\mathcal H^2$.

Let us introduce the following class of symbols.
We say that a holomorphic function $\varphi:\CC_0\to\CC$
belongs to $\GG$
if it can be written $\varphi(s)=c_0s+\psi(s)$
where $c_0$ is a positive integer and $\psi$ is a Dirichlet series such that, for all $\veps>0,$
$\psi$ converges uniformly on $\CC_\veps.$

\smallskip

Our first main theorem now reads:
\begin{theorem}\label{thm:lfm}
 Let $\varphi(s)=c_0s+\sum_{k=1}^{+\infty}c_k k^{-s}=c_0 s+\psi(s)\in\GH$.
 \begin{itemize}
  \item[(i)] If $c_0=1,$ there exists a unique, up to an additive constant, $u\in\GG$ with $\textrm{char}(u)=1$ such that $u\circ\varphi=u+c_1$.
  Moreover, if there exists $\delta>0$ such that $\varphi(\CC_0)\subset\CC_\delta,$ we may choose $u\in\GH$.
  \item[(ii)] If $c_0>1,$ there exists a unique $u\in\GH$ with $\textrm{char}(u)=1$ such that $u\circ\varphi=c_0 u$.
 \end{itemize}
 In both cases, if $\varphi$ is injective, then $u$ is injective.
\end{theorem}

It should be said that a result related to assertion (i) of this theorem was recently obtained in \cite{cgl} in the context
of semigroups of operators on Hilbert spaces of Dirichlet series.

\smallskip

We then deduce several interesting applications. Indeed Theorem \ref{thm:lfm} is helpful to compute
the spectrum of composition and weighted composition operators (see e.g.  Theorem \ref{sfe-thm}) as well as to give
non trivial examples of cyclic composition operators (see Theorem \ref{cor:cyclic} and Example \ref{ex:cyclic}).
We will also use it to investigate thoroughly the commutant of a composition operator.

\subsection{Commutant of composition operators}

Denote by $\mathcal{B}(H)$ the algebra of all bounded linear operators on a Hilbert space $H$.
Given any $A\in \mathcal{B}(H)$, recall that the commutant of $A$ is defined as the family of all operators that commute with $A$,
that is,
\begin{equation*}
  \{A\}':=\{T\in\mathcal{B}(H):TA=AT\}.
\end{equation*}
Clearly, $\{A\}'$ is a closed subalgebra of $\mathcal{B}(H)$ in the weak operator topology $\sigma$.
Let $\alg (A)$ stand for the unital algebra generated by the operator $A$, that is,
\begin{equation*}
  \alg (A):=\{p(A):p\ \mbox{is a polynomial}\}.
\end{equation*}
It is easy to see that $\overline{\alg (A)}^{\sigma}$ is a commutative algebra such that $\overline{\alg (A)}^{\sigma}\subset \{A\}'$.
Then we say that an operator $A$ has a minimal commutant provided that
\begin{equation*}
  \overline{\alg (A)}^{\sigma}=\{A\}'.
\end{equation*}

Recently, M. Lacruz, F. Leon-Saavedra, S. Petrovic and L. Rodriguez-Piazza
\cite{llsr18}
investigated the minimal commutant property for a composition operator $C_{\varphi}$ on the Hardy space $H^{2}(\mathbb{D})$.
They got a complete answer provided $\varphi$ is a linear fractional map.
Precisely, if $\varphi$ is a linear fractional self-map of the unit disk $\mathbb{D}$, then on $H^{2}(\mathbb{D})$, $C_{\varphi}$ has a minimal commutant
(i.e., $\overline{\alg (C_{\varphi})}^{\sigma}=\{C_{\varphi}\}'$)
exactly when $\varphi$ is a non-periodic elliptic automorphism,
or a parabolic non-automorphism, or a loxodromic map.

We shall investigate the same property for composition operators on $\mathcal H^2$.
Since the results on the unit disc essentially deal with linear fractional maps,
we do not expect to get a complete characterization of the symbols $\varphi\in\GH$ such that $C_\varphi$ has a minimal commutant.
However, since we have no analogues to linear fractional symbols, we are looking for non trivial examples.
It turns out that we shall get very different results provided $c_0=1$ or $c_0>1$.
In the latter case, the following very general result says that in many cases,
$C_\varphi$ fails to have a minimal commutant.

\begin{theorem}\label{thm:c02}
 Let $\varphi\in\GH$ with $\textrm{char}(\varphi)\geq 2$. Assume that there exists $\veps>0$ such
 that $\varphi(\CC_0)\subset\CC_{\veps}$.
 Then $C_\varphi$ fails to have a minimal commutant.
\end{theorem}

The case of symbols with characteristic equal to $1$ is completely different.
We shall exhibit a large class of symbols $\varphi\in\GH$ with characteristic $1$ such that
$C_\varphi$ has a minimal commutant (see Theorem \ref{thm:characteristic1} and Example \ref{ex:characteristic1}).

It turns out that our methods of proof have also consequences for composition operators on $\mathbb D$
and we shall give new examples of composition operators on $H^2(\DD)$ with, or without, the minimal
commutant property.

\subsection{Organization of the paper}
\begin{itemize}
    \item In Section \ref{sec:schroeder}, we first introduce the Fr\'echet space of uniformly convergent Dirichlet series $\hinfplus$
    which has been recently studied in \cite{Bon18} and which will play the role of $H(\DD)$
    for Dirichlet series. We also introduce a class of symbols $\GD$ which is more general that $\GG$
    and we give some properties of symbols in $\GD$ which will be useful throughout the paper.
    We finally prove Theorem \ref{thm:lfm} about the existence of a K\"{o}nigs map for symbols in $\GH$.
    \item In Section \ref{sec:spectrum}, following the recent works \cite{ABC,ABCC}, we use the existence of a K\"{o}nigs map to fully characterize the spectrum of (weighted) composition operators
    acting on $\mathcal \hinfplus$ when the symbol belongs to $\GH$ and is not an automorphism.
    This also yields new informations on the spectrum of composition operators on $\mathcal H^2.$
    \item In Section \ref{sec:cyclic}, we show how the existence of a K\"{o}nigs map
    and its properties can be helpful to construct cyclic composition operators on $\mathcal H^2.$ In particuler, we provide nontrivial examples.
    \item Section \ref{sec:affine} is devoted to the study of the minimal commutant property of composition operators with linear symbols.  In this case, we get a complete description.
    \item  In Section \ref{sec:c02}, we study the minimal commutant property for composition operators with characteristic greater than or equal to $2$. In particular, we prove Theorem \ref{thm:c02} and we also handle the case of affine symbols.
    \item In Section \ref{sec:c01}, we are concerned with the minimal commutant property of composition operators with characteristic $1$. We give a sufficient condition for such operators to have the minimal commutant property and we exhibit non trivial examples.
    \item In Section \ref{sec:disc}, we go to the unit disc, and show how our methods shed new light to the minimal commutant property for composition operators acting on $H^2(\mathbb D)$. In particular, we handle symbols satisfying $\varphi(0)=\varphi'(0)=0$
    and give a variant of K\"{o}nigs maps for them.
    This allows us to prove that the composition operator induced by this class of symbols do not have the minimal commutant property.
    We also exhibit example of composition operators with the minimal commutant property which are not induced by linear fractional maps.

    \item In the last section, we discuss the double commutant property for composition operators with affine symbols. This leads us to solve a general problem: when does the direct sum of weighted shifts possess the double commutant property?
\end{itemize}


\section{K\"{o}nigs maps}\label{sec:schroeder}

\subsection{Two classes of functions}

For our purposes it is convenient to introduce two classes of functions. We first need a class of symbols more general than $\GG.$
Most of our general results will be true for this class of symbols. We say that a holomorphic function belongs to $\GD$ if it is defined on some domain $U$
containing a half-plane $\CC_\sigma$ and if, on $\CC_\sigma$, it can be written $\varphi(s)=c_0s +\psi(s)$ where $c_0$ is a positive integer
and $\psi$ is a Dirichlet series which converges in $\CC_\sigma$. In particular, $\GH$ is the set of functions $\varphi\in\GD$ such that
$\varphi$ is defined on $\CC_0$ and $\varphi(\CC_0)\subset\CC_0$.

We will also need a Fr\'echet algebra of Dirichlet series which will play the role of $H(\DD)$ in our context. Such a Fr\'echet space has been introduced
and studied in \cite{Bon18}. We denote by $\hinfplus$ the space of all analytic functions which are bounded on $\CC_\veps$
for all $\veps>0$ and that can be represented as a convergent Dirichlet series in $\CC_0$. It is endowed with the metrizable locally convex
topology defined by the seminorms
$$P_\veps(f)=\sup_{s\in\CC_\veps}|f(s)|,\ f\in\mathcal \hinfplus.$$
By a fundamental theorem of Bohr, $\hinfplus$ coincides with the space of all Dirichlet series
which are uniformly convergent in each $\CC_\veps,$ $\veps>0.$ It contains $\mathcal H^\infty,$ the set of functions in $\mathcal D$ which define a bounded function on $\CC_0$.

Several analogues of $H(\mathbb D)$ could be considered.
For instance, one could study composition operators on $\mathcal D_a(\CC_0)$, the set of Dirichlet series $\sum_{n\geq 1}a_n n^{-s}$ such that $\sum_{n\geq 1}|a_n|n^{-\sigma}$ is convergent for all $\sigma>0$.
We have chosen to work with $\hinfplus$ for two reasons:
it is an algebra (this is not the case of $\mathcal D_a(\CC_0)$) and the topology of uniform convergence on right half-planes seems the right analogue of the topology of convergence on compacta of $\CC_0.$

\subsection{Useful lemmas}

We shall need, throughout the paper, a couple of lemmas regarding maps in $\GD$. Some of them will be easily deduced from standard results
on Dirichlet series. We first recall the following result on the behaviour of a Dirichlet series near infinity (see \cite[Lemma 3.1]{gh}).

\begin{lemma}\label{lem:dirinfini}
 Let $f(s)=\sum_{n=m}^{+\infty}a_n n^{-s}$ be a Dirichlet series convergent in some half-plane
 such that $a_m\neq 0$. Then $m^s f(s)\to a_m$ uniformly as $\A (s)\to+\infty$.
\end{lemma}

Hence, if $\psi\in\mathcal D,$ the real part of $\psi(s)$ is bounded below as $\A(s)$ tends to $+\infty$.
As a consequence, we easily get the following mapping properties of the maps belonging to $\GD$.
\begin{lemma}\label{lem:gdinfini}
 Let $\varphi\in \GD$. For all $\sigma_0>0,$ there exists $\sigma_1>0$ such that $\varphi(\CC_{\sigma_1})\subset\CC_{\sigma_0}$.
\end{lemma}

We have a kind of converse of this result.

\begin{lemma}\label{lem:reverseinclusion}
Let $\varphi:U\to\mathbb C$ belonging to $\GD$. For all $\sigma_1\in\mathbb R$, there exists $\sigma_0>0$  such that $\CC_{\sigma_0}\subset \varphi(U\cap \CC_{\sigma_1})$.
\end{lemma}

\begin{proof}
Enlarging $\sigma_1$ if necessary we can always assume that $\CC_{\sigma_1}\subset U$ and that on $\CC_{\sigma_1}$, one can write $\varphi(s)=c_0s+\psi(s)=c_0 s+c_1+\psi_0(s)$ with $|\psi(s)|\leq M$ for some $M>0$ independent of $s\in\CC_{\sigma_1}$.
 Let us set $\sigma_0=c_0\sigma_1+\A(c_1)+M$ and let us consider $w\in\CC_{\sigma_0}$. We define
\begin{align*}
f(s)&=c_0s+c_1+\psi_{0}(s)-w\\
g(s)&=c_0s+c_1-w.
\end{align*}
Let $K=[\sigma_1,A]\times [-\tau,\tau]$ for some large $A$ and $\tau$. Provided $\tau$ and $A$ are large enough, $g$ vanishes in $K$.
On the other hand, for $s\in\partial K,$
$$|f(s)-g(s)|\leq M$$
whereas, for $s=\sigma_1+it,$
\begin{align*}
|g(s)|&\geq \A(w-c_0s-c_1)\\
&>\sigma_0-c_0\sigma_1-\A(c_1)=M.
\end{align*}
We then adjust $A$ and $\tau$ large enough so that this last inequality remains true on the other sides of $K$. The result then follows from an application of Rouch\'{e}'s theorem.
\end{proof}

Another consequence of Lemma \ref{lem:dirinfini} is the injectivity near infinity of maps in $\GD$.
\begin{lemma}\label{lem:injectivity}
Let $\varphi\in\GD$. There exists $\sigma_1\in\mathbb R$ such that  $\varphi_{|\CC_{\sigma_1}}$ is one-to-one.
\end{lemma}
\begin{proof}
We again write $\varphi(s)=c_0s+\psi(s)$ and consider $\sigma_1\in\mathbb R$ large enough so that 
$|\psi'(s)|\leq 1/2$ on $\CC_{\sigma_1}$. Then if $s_0,s_1\in \CC_{\sigma_1},$ the mean value theorem applied to $\psi$ implies that
$$|\varphi(s_0)-\varphi(s_1)|\geq \frac 12|s_0-s_1|$$
so that $\varphi$ is one-to-one on $\CC_{\sigma_1}$.
\end{proof}

We will need that the class $\GD$ is stable under uniform convergence on half-planes.
Again, this will follow from the corresponding result on Dirichlet series.

\begin{lemma}\label{lem:uniformdir}
 Let $(D_n)$ be a sequence of $\mathcal D$ and assume that $(D_n)$ converges uniformly to $D$ on some $\CC_\sigma$, $\sigma>0$.
 Then $D\in\mathcal D$.
\end{lemma}
\begin{proof}
 There exists $N\geq 1$ such that, for all $n,p\geq N,$ $\sup_{s\in\CC_\sigma}|D_n(s)-D_p(s)|\leq 1.$
 Now, $D_N$ can be represented by a Dirichlet series in some half-plane $\CC_{\sigma'}$, $\sigma'\geq \sigma$.
 In particular, it is bounded in $\overline{\CC_{\sigma'+2}}$ since the abscissa of absolute convergence of $D_N$ is not
 greater than $\sigma'+1$. Hence, for all $n\geq N,$ $D_n$ is bounded in $\overline{\CC_{\sigma'+2}}$. Since
 $\mathcal H^\infty$ is a closed subspace of $H^\infty(\CC_0)$ (see  \cite[Lemma 18]{b1} or \cite[Proof of Theorem 1.17]{DGMS}),
 we conclude that $D$ belongs to $\mathcal D$.
\end{proof}

\begin{lemma}\label{lem:uniformgd}
 Let $(\varphi_n)$ be a sequence of $\GD$ such that $(\varphi_n)$ converges uniformly to $\varphi$ on some $\CC_\sigma,$ $\sigma>0.$
 Then $\varphi\in\GD$.
\end{lemma}
\begin{proof}
 We observe that if $\varphi_n$ and $\varphi_m$ do not have the same characteristic, then $\sup_{s\in\CC_\sigma}|\varphi_n(s)-\varphi_m(s)|=+\infty$.
 Therefore, for $n$ large enough, all the maps $\varphi_n$ have the same characteristic $c_0$.
 The lemma follows now from Lemma \ref{lem:uniformdir} applied to the sequence $(\varphi_n-c_0s)$.
\end{proof}

We have a similar result for sequences in $\GG.$

\begin{lemma}\label{lem:uniformgh}
Let $(\varphi_n)$ be a sequence in $\GG$ such that $(\varphi_n)$ converges uniformly to $\varphi$ on every $\CC_\veps,$ $\veps>0$.
Then $\varphi\in\GG.$
\end{lemma}
\begin{proof}
As above, we may write $\varphi_n(s)=c_0 s+\psi_n(s)$ with $\psi_n\in\hinfplus$. Since $\hinfplus$
is closed under uniform convergence over every half-plane $\CC_\veps,$ $\veps>0,$
$(\psi_n)$ converges uniformly on each $\CC_\veps$, $\veps>0$, to $\psi\in\hinfplus$.
Therefore $\varphi(s)=c_0 s+\psi(s)\in\GG.$
\end{proof}

The maps of the class $\GD$ are designed so that $f\circ\varphi\in \mathcal D$ for all $f\in\mathcal D$ and all $\varphi\in\GD$.
We can also compose two maps of $\GD$ as the following lemma indicates.

\begin{lemma}\label{lem:compogd}
Let $u\in \mathcal{D}\cup\GD$ and $\varphi\in\GD$ and let us write them
\begin{align*}
 u(s)&=d_0 s+d_1+\sum_{n\geq 2}d_n n^{-s}\\
 \varphi(s)&=c_0 s +c_1+\sum_{n\geq 2}c_n n^{-s}.
\end{align*}
Then $u\circ\varphi$ is in $\mathcal{D}\cup\GD$. Moreover, one can write
$$u\circ\varphi=d_0c_0 s +d_0c_1+d_1+\sum_{n\geq 2}a_n n^{-s}$$
where, for $n\geq 2,$
$$a_n=d_0 c_n+\sum_{k^{c_0}|n,\ k \geq 1}\alpha(k,n)d_k$$
where $\alpha(k,n)$ depends only on $\varphi$ and $\alpha(n,n^{c_0})=n^{-c_1}$.
\end{lemma}
\begin{proof}
 We follow the proof of \cite[p.318]{gh}. By Lemma \ref{lem:gdinfini}, the map $u\circ\varphi$ is well defined on some half-plane $\CC_\sigma$.
 We then write, for $\A(s)$ large enough,
 \begin{align*}
  u\circ\varphi(s)&=d_0\left(c_0 s+c_1+\sum_{n\geq 2}c_nn^{-s}\right)+d_1\\
  &\quad\quad+\sum_{k\geq 2}d_k k^{-c_0 s }k^{-c_1}\prod_{l=2}^{+\infty}\sum_{j=0}^{+\infty}\frac{(-c_l \log k)^j l^{-js}}{j!}.
 \end{align*}
The lemma will be proved if we can prove the absolute convergence of the Dirichlet series
obtained by expanding the previous equality.  This follows from the convergence of
\begin{align*}
 \sum_{k\geq 2}|d_k| k^{-c_0 \A(s) }k^{-\A(c_1)}\prod_{l=2}^{+\infty}\sum_{j=0}^{+\infty}\frac{|c_l| (\log k)^j l^{-j\A(s)}}{j!}\\
 =\sum_{k\geq 2}|d_k| k^{-c_0\A(s)}k^{-\A(c_1)}\exp\left(\log(k)\sum_{l=2}^{+\infty}|c_l|l^{-\A(s)}\right)
\end{align*}
which is due to the absolute convergence of all the involved Dirichlet series in some remote half-plane (which in particular implies
that $\sum_{l\geq 2}|c_l| l^{-\A(s)}$ is bounded provided $\A(s)$ is large enough).
\end{proof}

As a corollary, we get the following result.

\begin{lemma}\label{lem:exponent}
 Let $\varphi\in\GD$, $\varphi(s)=c_0 s+c_1+\psi_{0}(s)$. Then $2^{-\varphi(s)}=\sum_{k\geq 1}b_k k^{-s}$
 with $b_k=0$ provided $2^{c_0}$ does not divide $k$. Moreover, $b_{2^{c_0}}=2^{-c_1}\neq 0.$
\end{lemma}

If we take into account the range and the domain of definitions of the maps, then we also get the following lemma
(which is also Proposition 2.13 in \cite{cgl}).

\begin{lemma}\label{lem:composition}
 Let $\varphi_1,\varphi_2\in\GH$. Then $\varphi_1\circ\varphi_2\in\GH.$
\end{lemma}

We are now interested in the reciprocal of a function $u\in \GD$.

\begin{lemma}\label{lem:inverse}
 Let $u:\Omega\to \CC$ be injective which belongs to $\GD$ with characteristic $1$, $u(s)=s+\psi(s).$ Then $u^{-1}:u(\Omega)\to \Omega$
 belongs to $\GD$.
\end{lemma}
\begin{proof}
 Let $\sigma_0>0$ be such that $| \psi|\leq 1$ and $|\psi'|\leq 1/2$ on $\overline{\CC_{\sigma_0}}$.
 Let $w\in\CC$ be such that $\A(w)\geq \sigma_0+1$ and define, for $s\in\overline{\CC_{\sigma_0}},$
 $$\phi_w(s)=s-\big(u(s)-w\big)=w-\psi(s).$$
 Then $\phi_w(\overline{\CC_{\sigma_0}})\subset \overline{\CC_{\sigma_0}}$ and $|\phi_w'|\leq 1/2$ on $\CC_{\sigma_0}$. It follows
 that the sequence $(v_n(w))_{n\geq 0}$ defined by $v_0(w)=w$ and, for all $n\geq 0,$
 $$v_{n+1}(w)=v_n(w)-\big(u(v_n(w))-w\big)=w-\psi(v_n(w))$$
 converges to some $v(w)\in\overline{\CC_{\sigma_0}}$ which satisfies
 $$\left\{
 \begin{array}{rcl}
  u\circ v(w)&=&w\\
  |v_{n+1}(w)-v(w)|&\leq&\left(\frac 12\right)^n |w-v(w)|\textrm{ for all }n\geq 1.
 \end{array}\right.$$
Therefore, $v=u^{-1}_{|\overline{\CC_{\sigma_0+1}}}.$ Moreover, letting $s=v(w),$ $\A(s)\geq\sigma_0$
so that
$$|w-v(w)|=|u(s)-s|\leq 1.$$
Hence, $(v_n)$ converges uniformly to $v$ on $\overline{\CC_{\sigma_0+1}}$. Now, the recurrence formula, Lemma \ref{lem:compogd} and Lemma \ref{lem:uniformgd}
ensure that $v\in\GD$.
\end{proof}

We finally need some results on the iterates of a function in $\mathcal G$.
For all $n\in\mathbb{N}$ we denote the $n$-th iterate of $\varphi$ by $\varphi^{[n]},$ that is,
 \begin{equation*}
   \varphi^{[n]}:=\varphi\circ\cdots\circ\varphi.
 \end{equation*}

\begin{lemma}\label{lem:iteration}
 Let $\varphi\in\GH,$ $\varphi(s)=c_0s+\psi(s)=c_0s+c_1+\psi_0(s)$.
\begin{itemize}
 \item[(i)] If $c_0\geq 2,$ then for all $\sigma>0,$ for all $n\geq 1,$ $\varphi^{[n]}(\CC_\sigma)\subset \CC_{c_0^n \sigma}.$
 \item[(ii)] If $c_0\geq 1$ and $\psi_0\neq 0,$ then for all $\sigma>0,$ for all $\alpha>0,$ there exists $N\in\mathbb N$ such that,
 for all $n\geq N,$ $\varphi^{[n]}(\CC_\sigma)\subset \CC_{\sigma+n(\A(c_1)-\alpha)}.$
\end{itemize}
\end{lemma}

\begin{proof}
When $c_0\geq 2,$ the result follows from an easy induction since $\A\psi\geq 0$ on $\CC_0.$ Let us assume that $c_0=1.$
Since $\psi$ is not constant, it follows from \cite[Theorem 4.2]{gh} that, for all $\sigma>0,$
 there exists $\nu(\sigma)>0$ such that $\psi(\overline{\CC_\sigma})\subset \overline{\CC_{\nu(\sigma)}}$. Therefore, by induction, for all $n\geq 1,$
  $$\varphi^{[n]}(\CC_\sigma)\subset \overline{\CC_{\sigma+n\nu(\sigma)}}.$$
Now, let $\alpha>0$ and observe that there exists $\sigma_0>0$ such that, for all $s\in\CC_{\sigma_0},$ $\A\psi_0(s)\geq -\alpha/2$
so that $\A \varphi(s)\geq \A (s)+\left(\A(c_1)-\frac\alpha 2\right).$ Fix $N_1\in\mathbb N$ such that $\varphi^{[N_1]}(\CC_\sigma)\subset\CC_{\sigma_0}.$ 
Then for all $n\geq N_1,$ 
$$\varphi^{[n]}(\CC_\sigma)\subset \CC_{\sigma_0+(n-N_1)(\A(c_1)-\alpha/2)}.$$
The result now follows easily.
\end{proof}

\subsection{Iteration and K\"{o}nigs map}

Let us proceed with the proof of our first main theorem.

\begin{proof}[Proof of Theorem \ref{thm:lfm}]
 (i) If $\varphi(s)=s+c_1,$ there is nothing to do. Therefore assume that $\varphi(s)=s+c_1+\psi_0(s)$ with $\psi_0\neq 0$ and observe that in that case $\A(c_1)>0$.
 We also know by Lemma \ref{lem:iteration} that, for all $\sigma>0,$ there exists $N\geq 1$ such that, for all $n\geq N,$
 \begin{equation}\label{eq:lfm}
  \varphi^{[n]}(\overline{\CC_\sigma})\subset \overline{\CC_{\sigma+n\A (c_1)/2}}.
 \end{equation}
Let $s\in\CC_0$. Then one easily shows by induction that
$$\varphi^{[n]}(s)=s+nc_1+\sum_{j=0}^{n-1}\psi_0(\varphi^{[j]}(s)).$$
Let us denote by $u_n$ the map defined on $\CC_0$ by
\begin{align*}
 u_n(s)&=\varphi^{[n]}(s)-\varphi^{[n]}(1)\\
 &=s-1+\sum_{j=0}^{n-1}\big(\psi_0(\varphi^{[j]}(s))-\psi_0(\varphi^{[j]}(1))\big).
\end{align*}
By Lemma \ref{lem:dirinfini}, there exist some half-plane $\CC_{\sigma_0}$ and $C>0$ such that, for all $w\in\CC_{\sigma_0}$, $|\psi_0(w)|\leq C2^{-\A(w)}$. Taking
\eqref{eq:lfm} into account, we obtain the normal convergence of the series $\sum_{j\geq 0}\big(\psi_0(\varphi^{[j]}(s))-\psi_0(\varphi^{[j]}(1))\big)$
on each half-plane $\CC_\veps,$ $\veps>0$. Hence by Lemma \ref{lem:uniformgh}, $(u_n)$ converges uniformly on each $\CC_{\veps}$, $\veps>0,$
to some holomorphic map $u:\CC_0\to\mathbb C$
with $u(s)=s-1+D(s)$ where $D\in\hinfplus$. Moreover, from $u_n\circ\varphi(s)=u_{n+1}(s)+\varphi^{[n+1]}(1)-\varphi^{[n]}(1)$,
we get by letting $n$ to $+\infty$ that $u\circ\varphi=u+c_1,$ since
$$\lim_{n\to+\infty}\left[\varphi^{[n+1]}(1)-\varphi^{[n]}(1)\right]=c_1+\lim_{n\to+\infty}\psi_0(\varphi^{[n]}(1))=c_1.$$

It remains to prove the uniqueness of $u$ up to an additive constant.
Assume that $u\in\GG$ has characteristic $1$ and satisfies $u\circ \varphi=u+c_1$.
Applying Lemma \ref{lem:compogd}, and writing $u(s)=s+d_1+\sum_{n\geq 2}d_n n^{-s},$
we get for all $n\geq 2,$
$$d_n=c_n+n^{-c_1}d_n+\sum_{\substack {k|n\\k<n}}\alpha(k,n)d_k$$
which immediately implies by induction that the sequence $(d_n)_{n\geq 2}$ is uniquely defined provided $d_1$ has been fixed.

Finally, if $\varphi$ is univalent,
each $u_n$ is univalent and Hurwitz's theorem implies that $u$ is either univalent or constant.
Since $u\in\GG$, $u$ is not constant.

If we assume that $\varphi(\CC_0)\subset \CC_\delta$ for some $\delta>0$, or even that $\varphi^{[p]}(\CC_0)\subset\CC_\delta$ for some $p\geq 1$ and some $\delta>0,$ then we get the normal convergence of  $\sum_{j\geq 0}\big(\psi_0(\varphi^{[j]}(s))-\psi_0(\varphi^{[j]}(1))\big)$
on the whole $\CC_0$. Hence, $\A(u_n)$ is uniformly bounded below on $\CC_0$, say by $a$, and one has just to replace $u$ by $u-a$
to get a map in $\GH$.

(ii) Observe that, for all $s\in\CC_0,$
$$\varphi^{[n]}(s)=c_0^n s +\sum_{j=0}^{n-1}c_0^{n-1-j}\psi(\varphi^{[j]}(s)).$$
Let us set $u_n(s)=\varphi^{[n]}(s)/c_0^n$. Then,
$$u_n(s)=s+\sum_{j=0}^{n-1}\frac{\psi(\varphi^{[j]}(s))}{c_0^{j+1}}.$$

Since $\psi$ is bounded on some half-plane $\CC_\theta$ and since by Lemma \ref{lem:iteration}
for all $\sigma>0$ there exists $J\in\mathbb N$ such that $\A(\varphi^{[j]}(s))>\theta$ for all $j\geq J$ and all $s\in\CC_\theta$,
the series $\sum_{j\geq 0}\psi(\varphi^{[j]}(s))/c_0^{j+1}$, hence the sequence $(u_n)$, converge uniformly on each half-plane $\CC_\sigma$, $\sigma>0.$
Denote by $u\in\mathcal G_{>0}$ the limit of $(u_n).$
The equality $u_n\circ\varphi=c_0 u_{n+1}$ leads to $u\circ\varphi=c_0 u.$

The uniqueness follows the same lines as above. Indeed, if $u\in\GH$ has characteristic $1$ and satisfies $u\circ\varphi=c_0 u,$
then writing $u(s)=s+d_1+\sum_{n\geq 2}d_n n^{-s},$
one must have
$$c_1+d_1=c_0d_1$$
and for all $n\geq 2,$
$$c_0d_n=c_n+\sum_{k^{c_0}|n}\alpha(k,n)d_k$$
which also implies that the sequence $(d_n)_{n\geq 1}$ is uniquely defined.
\end{proof}

In the sequel, we will always call any solution of $u\circ\varphi=u+c_1$ (if $c_0=1$) or $u\circ\varphi=c_0 u$
(if $c_0>1$), with $u\in\GG,$
the {\bf K\"{o}nigs map} of $\varphi$ (even if it is defined up to an additive constant
in the first case). Moreover, if $\A(u)$ is bounded below, we will always assume that $u\in\GH$.

\begin{remark}\label{rem:iteration}
 In the proof of Theorem \ref{thm:lfm}, we could alternatively defined $u_n$ by $u_n(s)=\varphi^{[n]}(s)-nc_1.$
\end{remark}

\begin{remark}
For several properties of the induced composition operator,
it will be important to know that the K\"onigs maps of a symbol $\varphi\in\GH$ with characteristic $1$ belongs to $\GH.$ We will show later (see Proposition \ref{prop:automaticcompactness}) that, provided $\varphi$ is univalent, this is equivalent to ask that $\varphi^{[p]}(\CC_0)\subset\CC_\delta$ for some $\delta>0$
and some $p\geq 1$.
\end{remark}

\begin{example}
Let $\varphi(s)=s+1-2^{-s}$. Then its K\"onigs map does not belong to $\GH$. Indeed, let $N\in\NN$. Since $\varphi^{[N]}(0)=0,$ there exists $\sigma_0\in(0,1)$ such that 
$\varphi^{[N]}(\sigma_0)\leq 1.$ Since $\varphi(\sigma)\leq \sigma+1$ for all $\sigma\in\RR,$ we get, for all $n\geq N$, $\varphi^{[n]}(\sigma_0)\leq (n-N)+1$ so that $u_n(\sigma_0)=\varphi^{[n]}(\sigma_0)-n\leq -N+1$ (we choose the normalization of Remark \ref{rem:iteration}). Taking the limit, $u(\sigma_0)\leq -N+1$ which shows that $\A (u)$ is not bounded below.
\end{example}

\section{Schr\"{o}der's equation and spectrum of composition operators} \label{sec:spectrum}

\subsection{Schr\"{o}der's equation and point spectrum of composition operators}

We can now state our theorem about the solutions of Schr\"{o}der's equation.
For $\varphi\in\mathcal G,$ denote by $\mathcal C_\varphi$ the induced composition operator on $\hinfplus$
(so that it is not confused with $C_\varphi$ acting on $\mathcal H^2$).

\begin{theorem}\label{sfe-thm}
 Let $\varphi(s)=c_0s +\sum_{k=1}^{+\infty}c_k k^{-s}\in\GH$ and let $u$ be its K\"{o}nigs map.
 \begin{itemize}
  \item[(i)] If $c_0>1,$ then $\sigma_p(\mathcal C_\varphi)=\{1\}$.
    Moreover, for any $f\in\hinfplus$, $f\circ\varphi=f$ if and only if $f$ is constant.
  \item[(ii)] If $c_0=1$ and $\varphi(s)\neq s+i\tau$ for some $\tau\in\mathbb R,$ then $\sigma_p(\mathcal C_\varphi)= \{m^{-c_1}:\ m\in\mathbb N\}$.
  Moreover, for each $m\in\NN$, $\ker(\mathcal C_\varphi-m^{-c_1})$ has dimension $1$ and is equal to $\textrm{span}(m^{-u})$.
  \item[(iii)] If $\varphi(s)=s+i\tau,$ $\tau\in\mathbb R^*,$ then $\sigma_p(\mathcal C_\varphi)=\{m^{-i\tau}:\ m\in\mathbb N\}.$
  Moreover
  \begin{itemize}
    \item[(a)] if $\tau=2k\pi/\log(m_0/n_0)$ with $k\in\ZZ^*$ and $m_0,n_0\in \mathbb N,$ $m_0\neq n_0,$ then for all $m\in\NN,$ $\ker(\mathcal C_\varphi-m^{-i\tau})$
    is spanned by $$\left\{n^{-s}:\ n\in\mathbb N\textrm{ and }n=m\left(\frac{n_0}{m_0}\right)^{\frac \ell k},\ \ell\in\mathbb Z\right\}.$$
     \item[(b)] if $\tau$ cannot be written $2k\pi/\log(m_0/n_0)$ with $k\in\ZZ^*,$ $m_0,n_0\geq 1$, $m_0\neq n_0$, then for all $m\in\NN,$ $\ker(\mathcal C_\varphi-m^{-i\tau})$
     has dimension $1$ and is equal to $\textrm{span}(m^{-s})$.
  \end{itemize}
 \end{itemize}
\end{theorem}

\begin{proof}
 (i) and (ii) In these cases, we have already observed during the proof of Theorem \ref{thm:lfm} that $\A \varphi^{[n]}(s)\to+\infty$ as $n\to+\infty$
 for any $s\in\CC_0$. If we assume that $f\circ\varphi=\lambda f$ for some $\lambda\in\mathbb C$ and some $f\in\hinfplus$, $f\neq 0$,
 we get that for all $s\in\CC_0,$ $f\circ\varphi^{[n]}(s)=\lambda^n f(s)$ and by letting $n$ to $+\infty,$
 we obtain that $(\lambda^n f(s))$ has to converge to $f(+\infty)$. If $f(+\infty)\neq 0,$ this implies that $\lambda=1$
 and $f$ is constant. If $f(+\infty)=0$, this implies $|\lambda|<1$.
 Moreover, writing $f(s)=a_mm^{-s}+\sum_{n>m}a_n n^{-s}$ with $a_m\neq 0,$ we get that $c_0=1$ otherwise the first term
 in $f\circ\varphi$ is a multiple of $m^{-c_0s}$, and when $c_0=1$, we must have $a_m m^{-c_1}=\lambda a_m.$

 Therefore, $\sigma_p(\mathcal C_\varphi)=\{1\}$ provided $c_0>1$ and $\sigma_p(\mathcal C_\varphi)\subset \{m^{-c_1}:\ m\in\mathbb N\}$
 provided $c_0=1$. To prove the converse inclusion, we claim that for all $m\geq 1,$
 $m^{-u}\in\hinfplus$. Indeed, $m^{-u}$ belongs to $\mathcal D$ and has an analytic extension to $\CC_0$. Moreover, we may write $u(s)=s+\psi(s)$ where $\psi\in\hinfplus$ so that, for all $\veps>0,$ there exists $M\in\RR$ such that $\A(u)\geq M$ on $\CC_\veps.$
 Hence, $m^{-u}(\CC_\veps)$ is bounded and the claim follows from Bohr's theorem. 
Finally, $\mathcal C_\varphi(m^{-u})=m^{-c_1}m^{-u}$.

 To prove that each $m^{-c_1}$ has multiplicity $1$, let $f(s)=\sum_{n\geq 2}a_n n^{-s}$ be such that $f\circ\varphi = m^{-c_1}f.$
 Then using Lemma \ref{lem:compogd}, we find
 $$a_1=m^{-c_1}a_1$$
 and for all $n\geq 2,$
 $$m^{-c_1}a_n =c_n+n^{-c_1}a_n+\sum_{\substack{k|n \\ k<n}} \alpha(k,n)a_k
 \iff (m^{-c_1}-n^{-c_1})a_n=c_n+\sum_{\substack{k|n \\ k<n}} \alpha(k,n)a_k.$$
 This defines uniquely the sequence $(a_n)_{n\geq 1}$ provided a value has been given to $a_m$.

(iii) Write $f(s)=\sum_{n\geq 1}a_n n^{-s}$ so that $f\circ\varphi(s)=\sum_{n\geq 1}a_n n^{-i\tau}n^{-s}$.
 Therefore, $\sigma_p(\mathcal C_\varphi)=\{m^{-i\tau}:\ m\in\mathbb N\}$ and for all $m\in\NN,$
 $m^{-s}$ is an eigenvector of $\mathcal C_\varphi$ corresponding to the eigenvalue $m^{-i\tau}$.
 To determine $\ker(\mathcal C_\varphi-m^{-i\tau})$, observe that it is sufficient to study when $n^{-\varphi(s)}=m^{-i\tau}n^{-s}$ for all $n\in\mathbb N$,
 namely when $m^{-i\tau}=n^{-i\tau}$.
 \begin{itemize}
  \item[(a)] Assume that $\tau=2k\pi/\log(m_0/n_0)$. Then it is easy to show that
  \begin{align*}
  n^{-i\tau}=m^{-i\tau}&\iff \exists \ell \in\mathbb Z, \tau\log(n/m)=2\ell\pi\\
  &\iff n=m\left(\frac{n_0}{m_0}\right)^{\frac \ell k}.
  \end{align*}
  \item[(b)] If there exists $n\neq m$ such that $n^{-i\tau}=m^{-i\tau},$ then $\tau=2k\pi/\log(m/n)$ for some $k\in\ZZ.$
 \end{itemize}
\end{proof}

\begin{remark}\label{rem:multiplicity}
 When $\varphi(s)=s+i\tau$ with $\tau=2k\pi/\log(m_0/n_0)$, for any $N\geq 1$, there are eigenspaces $\ker(\mathcal C_\varphi-m^{-i\tau})$
 with dimension greater than or equal to $N$, e.g. if $m=m_0^{N}$. In that case, $n^{-s}\in\ker(\mathcal C_\varphi-m^{-i\tau})$
 for $n=m_0^{N-l}n_0^l$, $l=0,\dots,N$.
\end{remark}

Schr\"{o}der's equation also gives some information on the spectrum of $C_\varphi$ acting on $\mathcal H^2$.
\begin{proposition}\label{prop:spectrumh2}
 Let $\varphi\in\GH$ with characteristic $1$ and let $u$ be its K\"{o}nigs map. Assume that $C_\varphi$ is compact.
 Then $\sigma(C_\varphi)=\{0\}\cup\{m^{-c_1}:\ m\in\NN\}$. Moreover, for each $m\geq 1,$ $\ker(C_\varphi-m^{-c_1})$
 has dimension $1$ and is spanned by $m^{-u}$.
\end{proposition}
\begin{proof}
 It was already proved in \cite{b2} that  $\sigma(C_\varphi)=\{0\}\cup\{m^{-c_1}:\ m\in\NN\}$. By compactness of
 $C_\varphi$, each $m^{-c_1}$ is an eigenvalue. Moreover, the proof of Theorem \ref{sfe-thm} shows that the subspace of solutions of the equation $f\circ\varphi=m^{-c_1}f$ with $f\in\mathcal D$ it at most one-dimensional. Since $m^{-u}$ solves this equation, $\ker(C_\varphi-m^{-c_1})=\textrm{span}(m^{-u}).$
\end{proof}

As a corollary we get the non trivial fact that, for a compact $C_\varphi$ with characteristic $1$, $m^{-u}$
belongs to $\mathcal H^2$ for all $m\geq 1$. Hence, the compactness of $C_\varphi$ gives information
on the K\"{o}nigs map of $\varphi$. Conversely, the following statement indicates that the knowledge of some properties of the K\"{o}nigs map
of $\varphi$ implies compactness of $C_\varphi$.

\begin{proposition}\label{prop:automaticcompactness}
 Let $\varphi(s)= s+\psi(s)\in\GH$ with characteristic $1$, univalent, and with $\psi$ non constant.
 Assume that the K\"{o}nigs map $u$ of $\varphi$ belongs to $\GH$. Then there exists $n\geq 1$ and $\delta>0$ such that $\varphi^{[n]}(\CC_0)\subset \CC_\delta$. In particular, $C_{\varphi^{[n]}}$ is compact.
\end{proposition}
\begin{proof}
 Writing $\psi(s)=\sum_{k\geq 1}c_k k^{-s},$ we know that for any $n\geq 1,$
\begin{equation}\label{eq:compactness}
u\circ\varphi^{[n]}=u+nc_1
\end{equation}
 and that $\A(c_1)>0$ since $\psi$ is not constant.
 Moreover applying Lemmas \ref{lem:gdinfini} and \ref{lem:inverse}, there exist $n\geq 1$ and $\delta>0$ such that $u^{-1}(\CC_{n\A(c_1)})\subset\CC_\delta.$
 Since $u\in\GH,$ \eqref{eq:compactness} yields $u\circ\varphi^{[n]}(\CC_0)\subset \CC_{n\A(c_1)}$ so that $\varphi^{[n]}(\CC_0)\subset \CC_\delta$
 and $C_{\varphi^{[n]}}$ is compact.
\end{proof}

\begin{remark}
If we only assume that $u\in\GG$ with characteristic $1$ is such that $m^{-u}$ belongs to $\GH$
for some integer $m\geq 1$,
then the proof of Proposition \ref{prop:spectrumh2} still shows that $\ker(C_\varphi-m^{-c_1})=\textrm{span}(m^{-u})$. 
\end{remark}

\subsection{Spectrum of composition operators}

We now discuss the spectrum of composition operators on $\hinfplus$. Our first step is to characterize the invertibility of $\mathcal C_\varphi.$ It turns out that this is more delicate than for $H(\DD)$ (see \cite{ABC}). Indeed, there are many multiplicative homomorphisms on $\hinfplus$ which are not a point evaluation at some point of $\CC_0$ because of the polydisc hidden behind $\hinfplus$ by the Bohr transform. We are inspired by \cite[Theorem 2.1]{Bou14}.

\begin{theorem}\label{thm:invertible}
Let $\varphi\in\GH.$ The following assertions are equivalent.
\begin{itemize}
\item[(i)] $\mathcal C_\varphi$ is invertible on $\hinfplus$.
\item[(ii)] There exists $\tau\in\RR$ such that $\varphi(s)=s+i\tau.$
\end{itemize}
\end{theorem}

\begin{proof}
We just have to show that $(i)\implies (ii)$ and we start with $\varphi\in\GH$ such that
$\mathcal C_\varphi$ is invertible. We first observe that $\varphi$ is one-to-one on $\CC_0.$
Indeed, let $s_0\neq s_1\in\CC_0$ and $D\in\hinfplus$ be such that $D(s_0)\neq D(s_1).$
Since $D=f\circ\varphi$ for some $f\in\hinfplus,$ we shall have $\varphi(s_0)\neq\varphi(s_1).$

We recall that $\varphi(i\tau)=\lim_{\sigma\to 0}\varphi(\sigma+i\tau)$ exists for a.e. $\tau\in\RR$ and we prove that $\A(\varphi(i\tau))=0$ for a.e. $\tau\in\mathbb R.$
By contradiction, assume that
$$E=\{\tau\in\mathbb R :\ \A(\varphi(i\tau))>0\}$$
has positive Lebesgue measure. Let also $f\in\hinfplus$ be such that $f\circ\varphi=2^{-s}$. There exist four real numbers $a$, $b$, $c$ and $d$ with
$0<a<b$ and $c<d$ such that
$$E_{a,b,c,d}=\{\tau\in\mathbb R:\ a\leq \A(\varphi(i\tau))\leq b,\ c\leq \B(\varphi(i\tau))\leq d\}$$
has positive Lebesgue measure. The set $\varphi(E_{a,b,c,d})$ cannot be finite (otherwise $\varphi$ would map a subset of $i\mathbb R$ having positive measure to a single point, forcing $\varphi$ to be constant) and therefore there exists $\tau_0\in E_{a,b,c,d}$ such that $f'(\varphi(i\tau_0))\neq 0,$ since $f$ is not constant. In particular, there exists a disc $D_1$ centered at $\varphi(i\tau_0)$ and contained in $\CC_0$ such that $f_{|D_1}$ is invertible; we denote by $f^{-1}$ its inverse.

By continuity of $2^{-s}$ up to the imaginary axis, $f\circ\varphi(i\tau_0)=2^{-i\tau_0}$ and thus there is a disc $D_2$ centered at $i\tau_0$ such that $f^{-1}\circ 2^{-s}$ is analytic on $D_2$ and maps $D_2$ into $D_1\subset \CC_0$. Since for $\A(s)>0$ and $s\in D_2,$ $f^{-1}(2^{-s})=\varphi(s),$ we get an analytic extension of $\varphi,$ that we shall call $\tilde\varphi$, to $\CC_0\cup D_2$, and which maps $\CC_0\cup D_2$ into $\CC_0.$

Let now $(\veps_n)_{n\in\mathbb N}$ be a sequence contained in $\{-1,+1\}$ and such that $g(s)=\sum_{n\geq 1}\frac{\veps_n}n n^{-s}$ admits $i\mathbb R$ as its natural boundary (see \cite{Qu80} for the existence of such a sequence).
This function belongs to $\hinfplus$ since the Dirichlet series converges absolutely
on each half-plane $\CC_\sigma,$ $\sigma>0.$ Therefore there is $g_1\in\hinfplus$ such that
$g_1\circ\varphi=g.$ The function $g_1\circ\tilde\varphi$ is then an analytic extension of $g$ to $\CC_0\cup D_2,$ which contradicts that $i\mathbb R$ is its natural boundary.

Therefore we have shown that $\varphi$ is a univalent self-map of $\CC_0$ such that $\A(\varphi(i\tau))=0$ for a.e. $\tau\in\mathbb R.$ If we conjugate it by the Cayley map, we find a univalent self-map of the unit disc which is inner. Hence, its conjugate is an automorphism of $\DD$ and $\varphi$ itself is an automorphism of $\CC_0.$ Since it fixes $+\infty$ and belongs to $\GH,$ one can write it $\varphi(s)=c_0 s+i\tau$
for some positive integer $c_0$ and some $\tau\in\mathbb R.$ We conclude because if $c_0\geq 2,$ $2^{-s}$ cannot be in the range of $\mathcal C_\varphi$ (see e.g. Lemma \ref{lem:exponent}) and thus $\mathcal C_\varphi$ is not invertible.
\end{proof}

We are now able to compute the spectrum of $\mathcal C_\varphi$ for $\varphi\in\GH$ not an automorphism.

\begin{theorem}\label{thm:spectrumcomposition}
Let $\varphi\in\GH,$ $\varphi(s)=c_0 s+\sum_{k\geq 1}c_k k^{-s},$ $\varphi$ not an automorphism.
\begin{itemize}
\item[(a)] If $c_0\geq 2,$ then $\sigma(\mathcal C_\varphi)=\{0,1\}.$
\item[(b)] If $c_0=1,$ then $\sigma(\mathcal C_\varphi)=\{0\}\cup\{m^{-c_1}:\ m\in\NN\}.$
\end{itemize}
\end{theorem}

\begin{proof}
(a) By Theorem \ref{sfe-thm} and Theorem \ref{thm:invertible}, we already know
that $\{0,1\}\subset\sigma(\mathcal C_{\varphi})$. Let us show that if $\lambda\notin \{0,1\},$ then $\mathcal C_\varphi-\lambda I$ is surjective (hence invertible by Theorem \ref{sfe-thm}). Let $f\in\hinfplus$ and write it $f=a+g,$ with $a\in \CC$ and $g\in\hinfplus,$ $g(+\infty)=0.$ Let us define, for $s\in\CC_0,$
$$G(s)=-\sum_{n\geq 0}\frac{g(\varphi^{[n]}(s))}{\lambda^{n+1}}.$$
Since, for all $n\geq 0$ and all $s\in\CC_0,$ $\A(\varphi^{[n]}(s))\geq c_0^n \A(s)$ and since, by Lemma \ref{lem:dirinfini}, there exist $\sigma_0>0$ and $M>0$ such that $|g(w)|\leq M2^{-\A(w)}$ if $\A(w)\geq \sigma_0,$ the series which defines $G$ converges normally, hence uniformly, on each $\CC_\veps,$ $\veps>0.$ Therefore $G$ belongs to $\hinfplus$. Let us finally set
$$F=\frac{a}{1-\lambda}+G\in\hinfplus.$$
It is easy to check that
$$F\circ\varphi-\lambda F=f.$$

(b) Again, we know that $\{0\}\cup\{m^{-c_1}:\ m\in\NN\}$ is contained in $\sigma(\mathcal C_\varphi)$. Let $\lambda\notin \{0\}\cup\{m^{-c_1}:\ m\in\NN\}$ and let us show that $\mathcal C_\varphi-\lambda I$ is surjective. Let $m\in \NN$ be such that
$m^{-\A(c_1)}<|\lambda|$ and let us set
\begin{align*}
E_m&=\textrm{span}(1,2^{-u},\dots,(m-1)^{-u})\\
F_m&=\overline{\textrm{span}}(n^{-s}:\ n\geq m)
\end{align*}
where $u$ is the K\"onigs map of $\varphi$ and the closure is taken in $\hinfplus.$
We have already observed that, for all $k\geq 2,$ $k^{-u}$ belongs to $\hinfplus$.
In particular, $E_m\subset\hinfplus$ and it is not hard to show that $\hinfplus=E_m\oplus F_m.$ First of all, observe that, for all $k\geq 2,$
\begin{equation}\label{eq:directsum}
k^{-u}=a_k k^{-s}+\sum_{l> k}a_{l,k}l^{-s}
\end{equation}
with $a_k\neq 0.$
\begin{itemize}
\item let $f\in E_m\cap F_m$ and let us write $f=b_1+\sum_{k=2}^{m-1} b_k k^{-u}.$
We shall have $b_1=0$ because any function in $F_m$ tends to $0$ as $\A(s)$ goes to $+\infty.$ If $f\neq 0,$ let $k\in\{2,\dots,m-1\}$ be the smallest integer such that $b_k\neq 0$. Then  by \eqref{eq:directsum}
$$\lim_{T\to+\infty}\frac 1{2T}\int_{-T}^T f(1+it)k^{-it}dt=\frac{a_k b_k}k$$
whereas, for any function in $F_m,$ the left hand side of this equality is always equal to zero. This is a contraction and we conclude that $E_m\cap F_m=\{0\}.$
\item let $f\in\hinfplus$ and write it $f=b_1+\sum_{k\geq 2}b_k k^{-s}.$ Then by \eqref{eq:directsum}, one can find by induction complex numbers $c_2,\dots,c_{m-1}$ such that
$$f(s)-b_1-\sum_{j=2}^{m-1}c_k k^{-u}\in F_m$$
so that $f\in E_m+F_m.$
\end{itemize}
Observe now that $E_m$ and $F_m$ are $\mathcal C_\varphi$-stable and that on $E_m,$
$\mathcal C_\varphi$ acts as a diagonal operator with respect to the basis $(1,2^{-u},\dots,(m-1)^{-u})$,  with coefficients $1,2^{-c_1},\dots,(m-1)^{-c_1}$. Therefore,
$\mathcal C_{\varphi}-\lambda I_{|E_m}$ is surjective. Finally, let $g\in F_m$ and let us define as above
$$G(s)=-\sum_{n\geq 0}\frac{g(\varphi^{[n]}(s))}{\lambda^{n+1}}.$$
Let $\alpha>0$ be such that $m^{-\A(c_1)+\alpha}<|\lambda|.$
Since $|g(w)|\leq M m^{-\A(w)}$ for some $M>0$ and all $w\in\CC_\sigma$ for some large $\sigma,$ and by applying Lemma \ref{lem:iteration}, we get the normal convergence of $G$ on each half-plane $\CC_\veps,$ $\veps>0$, so that $G\in\hinfplus$ satisfies $C_\varphi(G)-\lambda G=g.$
\end{proof}

\begin{question}
What is the spectrum of $\mathcal C_\varphi$ provided $\varphi(s)=s+i\tau$? Observe that this is already a delicate question for composition operators induced by a rotation on $H(\DD)$ (see \cite{Bon20}).
\end{question}

\subsection{Point spectrum of weighted composition operators}

Let $D\in\mathcal \hinfplus$ and $\varphi\in\GH$. We now consider the weighted composition operator
$W_{D,\varphi}$ defined on $\hinfplus$ by $W_{D,\varphi}=D\cdot f \circ\varphi$. Observe that this is a self-map of $\hinfplus$
since this last space is an algebra.

\begin{theorem}
 Let $D\in\hinfplus\backslash\{0\}$ and $\varphi\in\GH$, $\varphi(s)=c_0 s +\sum_{n\geq 1}c_nn^{-s}$, $c_0\geq 1$.
 \begin{itemize}
  \item[(a)] If $c_0\geq 2$ and $D(\infty)\neq 0,$ $\sigma_p(W_{D,\varphi})=\{D(\infty)\}$;
  \item[(b)] If $c_0=1$ and $D(\infty)\neq 0,$ $\sigma_p(W_{D,\varphi})=\{D(\infty)m^{-c_1}:\ m\in\NN\}$;
  \item[(c)] If $D(\infty)=0,$ $\sigma_p(W_{D,\varphi})=\varnothing$.
 \end{itemize}
\end{theorem}

\begin{proof}
 Let us write $D(s)=D(\infty)+\sum_{k\geq 2}a_k k^{-s}$. Let $\lambda\in\sigma_p(W_{D,\varphi})$ and let $f(s)=\sum_{k\geq 1}a_k k^{-s}$
 be a nonzero related eigenvector. Let $m$ be the smallest integer such that $a_m\neq 0$. Then looking at the coefficients of $W_{D,\varphi}(f)$,
 we have $m^{c_0}=m$ and $D(\infty)m^{-c_1}=\lambda$. This shows the first inclusion in cases (a) and (b). If $D(\infty)=0,$
 then we shall have $\lambda=0$ but this is impossible since $W_{D,\varphi}$ is injective provided $D\neq 0$.

 Conversely, let us assume that $D(\infty)\neq 0$. Since $\A(\varphi^{[n]}(1))$ goes to $+\infty,$
 there exists $n_0\in\NN$ such that, for $n\geq n_0,$ $D(\varphi^{[n]}(1))\neq 0.$ Let us define, for $n\geq n_0,$ $s\in\CC_0,$
 $$f_n(s)=\prod_{k=0}^{n_0-1}D(\varphi^{[k]}(s)) \prod_{k=n_0}^n \frac{D(\varphi^{[k]}(s))}{D(\varphi^{[k]}(1))}$$
 and let us show that $(f_n)$ converges uniformly on each half-plane $\CC_\sigma$, $\sigma>0$.
 Indeed let $\nu(\sigma)>0$ be such that $\varphi^{[n]}(\CC_\sigma)\subset \CC_{\sigma+n\nu(\sigma)}$ (see Lemma \ref{lem:iteration}).
 Then for all $s\in\CC_\sigma,$ 
 \begin{align*}
  \left|\frac{D(\varphi^{[n]}(s))}{D(\varphi^{[n]}(1))}-1\right|&=\left|\frac{D(\varphi^{[n]}(s))-D(\varphi^{[n]}(1))}{D(\varphi^{[n]}(1))}\right|\\
  &\leq C2^{-n\nu(\sigma)}
 \end{align*}
 by Lemma \ref{lem:dirinfini}. Let $f\in\hinfplus$ be the limit of $(f_n)$. Since for $n\geq n_0$ and $s\in\CC_0,$
\begin{align*}
W_{D,\varphi}(f_n)(s)&=
D(s)\prod_{k=0}^{n_0-1}D(\varphi^{[k+1]}(s)) \prod_{k=n_0}^n \frac{D(\varphi^{[k+1]}(s))}{D(\varphi^{[k]}(1))}\\
  &=D(\varphi^{[n+1]}(1))f_{n+1}(s),
 \end{align*}
 we get, letting $n$ to $+\infty$, $W_{D,\varphi}(f)=D(\infty) f.$ When $c_0=1$ and $m\in\NN,$
 let $g\in\hinfplus$ be such that $\mathcal C_\varphi(g)=m^{-c_1}g$. Then $fg\in\hinfplus$ and
 $$W_{D,\varphi}(fg)=W_{D,\varphi}(f)\mathcal C_\varphi(g)=D(\infty)m^{-c_1}fg.$$
\end{proof}

\begin{remark}
 Let us mention some consequences of the proof. The function $f$ defined above satisfies $f(+\infty)\neq 0$
 (this follows from the normal convergence of the product on $\CC_\sigma\cup\{+\infty\}$). In particular,
 there exists $h\in\hinfplus$ such that $W_{D,\varphi}(h)=D(\infty)h$ and $h(\infty)=1$.
 When $c_0=1,$ the function $hm^{-u},$ where $u$ is the K\"onigs function of $\varphi,$
 starts with $m^{-s}$ and satisfies
 $$W_{D,\varphi}(hm^{-u})=D(\infty)m^{-c_1} hm^{-u}.$$
\end{remark}

\subsection{Spectrum of weighted composition operators}
Arguing as above, we can now describe the spectrum of weighted composition operators on $\hinfplus$.

\begin{theorem}\label{thm:spectrumweighted}
 Let $D\in\hinfplus\backslash\{0\}$ and $\varphi\in\GH,$ $\varphi(s)=c_0s+\sum_{n\geq 1}c_n n^{-s},$ $\varphi$ is not an automorphism.
 \begin{itemize}
  \item[(a)] If $c_0\geq 2$ and $D(\infty)\neq 0,$ $\sigma(W_{D,\varphi})=\{0,D(\infty)\}.$
  \item[(b)] If $c_0=1$ and $D(\infty)\neq 0,$ $\sigma(W_{D,\varphi})=\{0\}\cup \{D(\infty)m^{-c_1}:\ m\in\NN\}.$
  \item[(c)] If $D(\infty)=0$, then $\sigma_p(W_{D,\varphi})=\{0\}.$
 \end{itemize}
\end{theorem}

\begin{proof}
We first consider the case where $D(\infty)\neq 0$.
 We begin by proving that if $W_{D,\varphi}$ is bijective, then $\varphi$ is an automorphism. So let us assume that $W_{D,\varphi}$ is bijective
 and let us show that, for all $\veps>0,$ there exists $\alpha>0$ such that $|D(s)|\geq\alpha$ for all $s\in\CC_\veps.$ If this was not the case,
 there would exist a sequence $(s_n)\subset\CC_\veps$ such that $D(s_n)\to 0.$ Let $f\in\hinfplus$ be such that $W_{D,\varphi}(f)=2+2^{-s}.$
 Then
 $$|D(s_n) f\circ\varphi(s_n)|=|2+2^{-s_n}|\geq 1.$$
 But $\varphi(s_n)\in\CC_\veps$ and $f$ is bounded on $\CC_\veps$; thus there is a contradiction.
 We then deduce that $1/D,$ which belongs to $\mathcal D$ since $D(\infty)\neq 0,$ is bounded on each $\CC_\veps$
 and therefore belongs to $\hinfplus.$ This yields that
 $$\mathcal C_\varphi=W_{1/D,\textrm{Id}}W_{D,\varphi}$$
 is also a bijection of $\hinfplus,$ hence that $\varphi$ is an automorphism.

 Therefore, since we assume that $\varphi$ is not an automorphism, we have shown that $0\in\sigma(W_{D,\varphi})$.
 Let us now prove (a). Let $\lambda\notin \{0,D(\infty)\}$ and let $f\in\hinfplus.$ Let also $h\in\hinfplus$ be such that
 $W_{D,\varphi}(h)=D(\infty) h$ and $h(\infty)=1.$ Let us consider $g=f-f(\infty)h,$ which satisfies $g(\infty)=0.$
 We also set, for $k\geq 0$ and $s\in\CC_0,$
 $$D_k(s)=D(s)D(\varphi(s))\cdots D(\varphi^{[k-1]}(s))$$
 and
 $$G(s)=-\sum_{n\geq 0}\frac{D_n(s)g(\varphi^{[n]}(s))}{\lambda^{n+1}}.$$
 Since, for any $\sigma>0,$ there exists $C>0$ such that $|D_n(s)|\leq C^n$ for all $n\in\mathbb N$
 and all $s\in\CC_\sigma,$ the argument of the proof of Theorem \ref{thm:spectrumcomposition} shows that
 $G$ belongs to $\hinfplus$ and that $W_{D,\varphi}(G)-\lambda G=g.$
 We set
 $$F=\frac{f(\infty)}{1-\lambda}h+G$$
 and verify that
 $$W_{D,\varphi}(F)-\lambda F=f(\infty)h+g=f,$$
 so that $W_{D,\varphi}-\lambda I$ is surjective.

 Let us now turn to (b). Let $\lambda\notin\{0\}\cup\{D(\infty)m^{-c_1}:\ m\in\NN\}$ and let us show that $W_{D,\varphi}-\lambda I$ is surjective.
 Let again $h\in\hinfplus$ be such that $W_{D,\varphi}(h)=h$ and $h(\infty)=1.$ For $m$ large enough, let us consider
 \begin{align*}
  E_m&=\textrm{span}(h,h2^{-u},\dots,h(m-1)^{-u})\\
  F_m&=\overline{\textrm{span}}(n^{-s}:\ n\geq m).
 \end{align*}
As before, $\hinfplus=E_m\oplus F_m$ and $E_m$, $F_m$ are $W_{D,\varphi}$-stable. Moreover, $W_{D,\varphi}$ acts as a diagonal operator on $E_m$
with respect to the basis $(h,h2^{-u},\dots,h(m-1)^{-u})$, with diagonal coefficients $(D(\infty),D(\infty)2^{-c_1},\dots,D(\infty)m^{-c_1}).$
Therefore, $W_{D,\varphi}-\lambda I_{|E_m}$ is surjective.

Let us now fix $g\in F_m$ and let us define
$$G(s)=-\sum_{n\geq 0}\frac{D_n(s)g(\varphi^{[n]}(s))}{\lambda^{n+1}}.$$
We just have to prove that $G\in\hinfplus.$ Let $\sigma>0$. Since $\A(\varphi^{[n]}(s))\to+\infty$
uniformly on $\CC_\sigma,$ there exists $C>0$ such that, for all $s\in\CC_\sigma,$
for all $n\geq 1,$
$$|D_n(s)|\leq C\big(|D(\infty)|+1\big)^n.$$
The uniform convergence of the series now follows from an argument similar to that of Theorem \ref{thm:spectrumcomposition},
choosing $m$ sufficiently large so that
$$m^{-\A(c_1)}\leq \frac{|\lambda|}{|D(\infty)|+1}.$$
Finally, when $D(\infty)=0,$ it is not hard to show that for any $g\in\hinfplus$ (even without assuming $g(\infty)=0$),
for any $\lambda\in\CC^*,$ the series
$$G(s)=-\sum_{n\geq 0}\frac{D_n(s)g(\varphi^{[n]}(s))}{\lambda^{n+1}}$$
converges in $\hinfplus$ and satisfies $W_{D,\varphi}G-\lambda G=g,$ so that $W_{D,\varphi}-\lambda I$
is surjective. This is not the case for $W_{D,\varphi},$ since any function $f$ in $\textrm{Ran}(W_{D,\varphi})$
satisfies $f(s)\to 0$ as $\A(s)\to+\infty.$
\end{proof}


\section{Cyclic composition operators} \label{sec:cyclic}

Recall that an operator $T$ on a Hilbert space $H$ is cyclic provided there exists a vector
$x\in H$ such that $\{P(T)(x):\ P\in\mathbb C[X]\}$ is dense in $H$.
Cyclic composition operators on $H^2(\DD)$ have been thoroughly studied in \cite{bs}.
It seems that nothing is known about this property on $\mathcal H^2$ (there is a mistake in the proof of Theorem 2.3 in \cite{yao}).
In this section, we initiate this study thanks to the existence of K\"onigs maps. We start with some general considerations.

\begin{proposition}
 Let $\varphi(s)=c_0s+\psi(s)\in \mathcal{G}.$ If $C_\varphi$ is cyclic on $\mathcal H^2$, then $c_0=1.$
\end{proposition}
\begin{proof}
 If $c_{0}\geq2$, then clearly $\{2^{-s},3^{-s}\}\subset (\ran (C_{\varphi}))^{\perp}$.
Since the orthogonal complement of the range of a cyclic operator has dimension at most one, then $C_{\varphi}$ is not cyclic.
If $c_0=0,$ then $\varphi\in\mathcal D$ and therefore $\varphi$ is not univalent on $\CC_{1/2}$ (see e.g. \cite[Theorem 2.3]{cgl}).
We can now argue as in \cite[Page 18]{bs}: by the open mapping theorem for holomorphic functions,
there exist infinitely many pairs of distinct points $s$ and $w$ in $\CC_{1/2}$ such that $\varphi(s)=\varphi(w)$.
For each such pair,
$$C_\varphi^*(K_s-K_w)=K_{\varphi(s)}-K_{\varphi(w)}=0$$
where $K_s$ is the reproducing kernel at $s\in\CC_{1/2}$. Therefore, $(\ran (C_{\varphi}))^{\perp}$
is infinite dimensional and $C_\varphi$ is not cyclic.
\end{proof}

\begin{question}
 The above proof shows that provided $C_\varphi$ is cyclic, $\varphi$ is univalent on $\CC_{1/2}$.
 Is it also univalent on $\CC_0$?
\end{question}

We then settle the case of linear symbols.

\begin{proposition}\label{linear:cyclic}
Let $\varphi(s)=s+c_1$ with $\A (c_{1})\geq0$.
Then $C_{\varphi}$ is cyclic on $\mathcal{H}^{2}$ if and only if
$c_{1}\neq i \frac{2k\pi}{\log(m/n)}$ with $k\in\mathbb{Z}\backslash \{0\}$,
$m, n\geq1$, $m\neq n$.
\end{proposition}
\begin{proof}
 This immediately follows from the characterization of cyclic diagonal operators, which can be found
 in \cite[Lemma 1]{Seu}: a diagonal operator $D$ on a Hilbert space $H$ given by $De_n=s_ne_n$
 where $(e_n)$ is an orthonormal basis of $H$ is cyclic if and only if $s_n\neq s_m$ for any $n\neq m.$
\end{proof}

We aim in this section to provide a sufficient condition for a composition operator on $\mathcal H^2$
to be cyclic and therefore to give nontrivial examples. We shall deduce the cyclicity of $C_\varphi$
from properties of the K\"{o}nigs map of $\varphi$.

\begin{lemma} \label{thm:cyclic}
Let $\varphi(s)=s+\sum_{k\geq 1}c_k k^{-s}\in\GH$ and $\A(c_1)>0$ be such that its K\"{o}nigs map $u$ belongs to $\GH$ and assume that $C_u$ has dense range.
 Then $C_\varphi$ is cyclic.
\end{lemma}

\begin{proof}
Let $\tau(s)=s+c_1$. Then $C_\tau$ is cyclic.
Let $f$ be a cyclic vector for $C_\tau$ and let $g\in\mathcal H^2,$
 $(P_n)\subset\CC[X]$ such that $(P_n(C_\tau)(f))$ goes to $g$ in $\mathcal H^2$.
 Then $(C_u(P_n(C_\tau)(f)))$ converges to $g\circ u$. Now, the relation $u\circ\varphi=\tau\circ u$
 yields
 $$C_u(P_n(C_\tau)(f))=P_n(C_\varphi)(f\circ u)$$
 so that $f\circ u$ is a cyclic vector for $C_\varphi$.
\end{proof}

\color{black}

\begin{remark}
We could also prove this lemma using a result of Clancey and Rogers \cite[Theorem 3]{cr} and the description of the spectrum given by Proposition \ref{prop:spectrumh2}.
\end{remark}

To prove that $C_u$ has dense range, one may appeal to the following lemma which is reminiscent from \cite[Theorem 3.4]{bs}.

\begin{lemma}\label{lem:denserange}
Let $\varphi(s)=s+\sum_{k\geq 1}c_k k^{-s}\in\GH$ be univalent and $\A(c_1)>0$.
Let $u$ be its K\"{o}nigs map and assume that it belongs to $\GH.$
If $C_\varphi$ has dense range, then $C_u$ has dense range.
\end{lemma}

\begin{proof}
 Let $m\in\mathbb N$ be such that $\CC_{m\A(c_1)}\subset u(\CC_0)$ and define, for $s\in\CC_0,$
 $$v(s)=u^{-1}(s+mc_1).$$
 Then $v\in\GH$. Moreover, since $u\circ \varphi^{[m]}=u+mc_1,$ $\varphi^{[m]}=v\circ u$. Hence
 \begin{align*}
  \textrm{Ran}(C_u)&\supset \{f\circ v\circ u:\ f\in\mathcal H^2\}\\
  &\supset \textrm{Ran}(C_{\varphi^{[m]}})
 \end{align*}
 which is dense since $C_\varphi$ has dense range.
\end{proof}

Combining Lemma \ref{thm:cyclic} and Lemma \ref{lem:denserange}, we get the following sufficient condition for a composition operator to be cyclic.

\begin{theorem}\label{cor:cyclic}
 Let $\varphi(s)=s+\sum_{k\geq 1}c_k k^{-s}\in\GH$ be univalent 
 and $\A(c_1)>0$. Assume that its K\"onigs maps belongs to $\GH$ and that $C_\varphi$ has dense range. Then $C_\varphi$ is cyclic.
\end{theorem}

To get examples of cyclic composition operators, we need a way to prove that $C_\varphi$ has dense range.
On the Hardy space of the unit disc,
this is linked to polynomial approximation in $A(\overline{\varphi(\DD)})$. For instance, if $\varphi(\DD)$
is a Jordan domain, then $C_\varphi$ has dense range. We do not have this kind of tools at our disposal.
We just give a sufficient condition which is enough to give nontrivial examples.

\begin{lemma} \label{lem:cyclic}
 Let $\varphi\in\GH$ be univalent with characteristic $1$. Assume that there exist $a<0<\sigma'< \sigma$ such that
 \begin{itemize}
  \item[(a)] $\varphi$ extends to a univalent map on $\CC_a$;
  \item[(b)] $\varphi(\CC_0)\subset\CC_\sigma$ and $\CC_{\sigma'}\subset \varphi(\CC_a)$.
 \end{itemize}
 Then $C_\varphi$ has dense range.
\end{lemma}
\begin{proof}
 Let $P$ be a Dirichlet polynomial. Since $\varphi^{-1}$ is defined on $\CC_{\sigma'}$
 and maps $\CC_{\sigma'}$ into $\CC_a,$
 $P\circ\varphi^{-1}$ belongs to $\mathcal D$ and is bounded on $\CC_{\sigma'}$. By Bohr's theorem, for all $\veps>0,$
 there exists a Dirichlet polynomial $Q$ such that, for all $w\in\CC_\sigma,$
 $$|P\circ\varphi^{-1}(w)-Q(w)|<\veps.$$
 Let now $s\in\CC_0$ and set $w=\varphi(s)\in\CC_\sigma$. Then
 $$|P(s)-Q\circ\varphi(s)|=|P\circ\varphi^{-1}(w)-Q(w)|<\veps.$$
 Since $\|f\|_2\leq \|f\|_\infty$ for all $f\in\mathcal H^\infty,$
 $$\|P-C_\varphi(Q)\|_2<\veps$$
 which shows that $C_\varphi$ has dense range.
\end{proof}

\begin{example}\label{ex:cyclic}
Let $\varphi(s)=s+c_1+c_2 2^{-s}$ with $c_1>0$ and $c_2\in \mathbb C$ satisfying
$$|c_2|< \frac{-\log\log(2)}{1+\log(2)}\textrm{ and }|c_2|<c_1.$$
 Then $C_\varphi$ is cyclic.
\end{example}

\begin{proof}
Let $a<0$.
It is not hard to show that $\varphi$ extends to $\CC_a$ with
 $$\varphi(\CC_0)\subset \CC_{c_1-|c_2|}\textrm{ and }\varphi(\CC_a)\supset \CC_{a+c_1+|c_2|2^{-a}}.$$
The conditions of Lemma \ref{lem:cyclic} are satisfied as soon as
 $$a+c_1+|c_2|2^{-a}<c_1-|c_2|$$
 and $c_1-|c_2|>0$. Now, the injectivity of $\varphi$ on $\CC_a$ is guaranteed by
$$(\log 2)2^{-a}\leq 1\iff a\geq \frac{\log\log(2)}{\log(2)}.$$
 The choice $a=\log\log(2)/\log(2)$ gives the condition of the example.
\end{proof}

\begin{question}
 Let $\varphi\in\GH$ be univalent. Under which conditions does $C_\varphi$ have dense range?
\end{question}

\section{Commutant of composition operators with linear symbols} \label{sec:affine}

A way (introduced by Worner in \cite{Wo02}) to prove that a composition operator does not have the minimal commutant property
is to study which multiplication operators commute with it. The multiplication operators on $\mathcal H^2$
are the maps $T_b:f\mapsto bf$ where $b\in\mathcal H^\infty$. Indeed, we may use the following proposition,
which is similar to \cite[Theorem 1.1]{llsr18}. For the sake of completeness, we include a proof.

\begin{proposition} \label{prop:notminimalmultiplier}
Let $\varphi\in\mathcal{G}$ and $b\in \mathcal{H}^{\infty}$.
Then $T_{b}\in \{C_{\varphi}\}'$ if and only if $b\circ\varphi=b$. If in addition $b$ is non-constant,
then $T_{b}\notin \overline{\alg(C_{\varphi})}^{\sigma}$.
\end{proposition}

\begin{proof}
If $b\circ\varphi=b$, then it is easy to see that $T_{b}C_{\varphi}f=C_{\varphi}T_{b}f$ for any $f\in\mathcal{H}^{2}$,
i.e., $T_{b}\in \{C_{\varphi}\}'$. Conversely, if $T_b\in\{C_{\varphi}\}',$ it is sufficient to apply the previous inequality with a non-vanishing Dirichlet series,
e.g. $2^{-s}$, to get $b\circ\varphi=b$.
In addition, if $b$ is non-constant, then clearly $b$ and $b^{2}$ are linearly independent. So there is $g\in\mathcal{H}^{2}$ such that $\langle b,g\rangle=0$ and $\langle b^{2},g\rangle=1$.
Define a linear functional $\Lambda: \mathcal{B}(\mathcal{H}^{2})\rightarrow\mathbb{C}$ by $\Lambda(T)=\langle Tb,g\rangle$.
Clearly, $\Lambda$ is continuous in the weak operator topology,
and thus the subspace $\ker \Lambda$ is closed in the weak operator topology.
Then $\overline{\alg(C_{\varphi})}^{\sigma}\in \ker \Lambda$.
On the other hand, $\Lambda(T_{b})=1$.
Therefore $T_{b}\notin \ker \Lambda$ and it follows that $T_{b}\notin\overline{\alg(C_{\varphi})}^{\sigma}$, as the desired result.
\end{proof}

Unfortunately, and contrary to what happens on the disc (see \cite{llsr18,llsr19}),
there are very few cases where we can exclude the minimal commutant property from Proposition \ref{prop:notminimalmultiplier}.
Indeed, as a consequence of Theorem \ref{sfe-thm}, it is easy to see that $\ker(\mathcal C_\varphi-1)$ has dimension greater than $1$
if and only $\varphi(s)=s+i\tau$ where $\tau=2k\pi/\log(m_0)$ for some $k\in\mathbb Z$ and some $m_0\geq 2$.
In this latter case, for any $\ell\in\mathbb Z$ such that $\ell/k\in\mathbb N$, the Dirichlet series $b(s)=m_0^{-\frac{\ell}k s}$ satisfies $b\circ\varphi=b$ and it belongs to $\mathcal H^\infty.$
We thus obtain the following corollary.

\begin{corollary}
 Let $\varphi(s)=s+i\tau$, $\tau\in\mathbb R$. Assume that there exist $k\in\mathbb Z^*$ and $m_0\geq 2$ such that $\tau=2k\pi/\log(m_0)$.
 Then $C_\varphi$ fails to have a minimal commutant.
\end{corollary}

Nevertheless, we will be able to settle completely the case of symbols $\varphi(s)=s+c_1$, with $\A(c_1)\geq 0$.
We first need a result on diagonal operators on a Hilbert space which shows that having the minimal
commutant property is related to an interpolation problem.

\begin{theorem}\label{thm:diagonal}
 Let $H$ be a Hilbert space with orthonormal basis $(e_n)_{n\in\mathbb N}$, let $(\lambda_n)$
 be a bounded sequence of complex numbers and let $D$ be the diagonal operator defined by $D(e_n)=\lambda_n e_n$.
 Assume that, for all bounded sequences of complex numbers $(\mu_n)$, there exists $M>0$ such that, for all $\veps>0,$
 for all $N\geq 1,$ there exists $P\in\mathbb C[X]$ such that
 $$\left\{
 \begin{array}{rcl}
  |P(\lambda_n)-\mu_n|&\leq&\veps,\ n=1,\dots,N\\
  \sup_{n\in\NN}|P(\lambda_n)|&\leq&M.
 \end{array}\right.$$
Then $D$ has a minimal commutant.
\end{theorem}

\begin{proof}
 Let us first observe that the assumptions imply that the sequence $(\lambda_n)$ is injective.
 Let $T\in\{D\}'$. There exists a sequence $(\mu_n)\subset\CC^\NN$, bounded by $\|T\|$,
 such that $Te_n=\mu_n e_n$ for all $n\in\NN$. Let, for $N\geq 1,$ $P_N\in\CC[X]$ be such that
 $\sup_{n\in\NN}|P_N(\lambda_n)|\leq M$ and $|P_N(\lambda_n)-\mu_n|\leq 1/N$
 for $n=1,\dots,N$. Then $(P_N(D))$ is a bounded sequence of $\mathcal B(H)$
 since
 $$\|P_N(D)\|=\sup_{n\in\NN} \|P_N(e_n)\|\leq M.$$
 Furthermore $P_N(D)(e_n)=P_N(\lambda_n)e_n$ goes to $Te_n$ for each $n\in\NN$.
 It follows easily that $(P_N(D))$ converges strongly to $T$.
\end{proof}

Using well-known results of interpolation, we then deduce a result on diagonal operators
which covers all cases associated to a composition operator with symbol $\varphi(s)=s+c_1$.

\begin{proposition}\label{prop:diagonal}
 Let $H$ be a Hilbert space with orthonormal basis $(e_n)_{n\in\mathbb N}$, let $(\lambda_n)$
 be a bounded sequence of complex numbers and let $D$ be the diagonal operator defined by $D(e_n)=\lambda_n e_n$.
 \begin{itemize}
  \item[(i)] If there exists $n\neq m$ such that $\lambda_n=\lambda_m,$ then $\{D\}'$ is not commutative and therefore $D$ fails to have a minimal commutant.
  \item[(ii)] If the sequence $(\lambda_n)$ is injective and convergent, then $D$ has a minimal commutant.
  \item[(iii)] If the sequence $(\lambda_n)$ is injective and contained in some circle, then $D$ has a minimal commutant.
 \end{itemize}
\end{proposition}

\begin{proof}
 (i) Let $E=\textrm{span}(e_n,e_m)$ and decompose $H$ as $E\oplus E^\perp$. Then for any $A\in \mathcal L(E)$, $A\oplus 0$ belongs to $\{D\}'$
 and since $\dim(E)\geq 2,$ $\mathcal L(E)$ is not commutative.

 (ii) Let $(\mu_n)$ be a bounded sequence of complex numbers, let $\veps>0$ and let $N\geq 1$. Let $\lambda=\lim_{n\to+\infty}\lambda_n$.
The set $K=\{\lambda_n:\ n\in\NN\}\cup\{\lambda\}$ is compact and has connected complement. Define $g$ on $K$ by $g(\lambda_n)=\mu_n$ if $n\leq N,$ $g(\lambda_n)=0$ if $n>N$ and $\lambda\notin\{\lambda_1,\dots,\lambda_N\},$
$g(\lambda_n)=\mu_m$ if $\lambda=\lambda_m$ with $m\in\{1,\dots,N\}.$
Then $g$ is continuous on $K$ and by Mergelyan's theorem (for example,
see \cite[Theorem 20.5]{rudin})
there exists $P\in\CC[X]$ such that
 $$|P(\lambda_n)-\mu_n|\leq \veps\textrm{ if }n\leq N,\ |P(\lambda_n)|\leq \|\mu\|_\infty\textrm{ if }n>N.$$
 By Theorem \ref{thm:diagonal}, $D$ has a minimal commutant.

(iii) The proof is similar,
replacing Mergelyan's theorem by Rudin-Carleson's interpolation theorem
(for example, see \cite{Carleson} and \cite{rudin56}).
Details are given in \cite[Proposition 3.2]{llsr18}.
\end{proof}

\begin{corollary}
 Let $\varphi(s)=s+c_1$ with $\A(c_1)\geq 0$, $c_1\neq 0$. Then $C_\varphi$ has a minimal commutant if and only if $c_1\neq i\tau$ where $\tau=2k\pi/\log(m_0/n_0)$
 with $k\in\ZZ^*,$ $m_0,n_0\geq 1$, $m_0\neq n_0$
\end{corollary}
\begin{proof}
 If $\A(c_1)>0$, we may apply (ii) of Proposition \ref{prop:diagonal}. If $\A(c_1)=0$ and $c_1$ is not equal to some $i\tau$ where $\tau=2k\pi/\log(m_0/n_0)$
 with $k\in\ZZ^*,$ $m_0,n_0\geq 1$, $m_0\neq n_0$, we may apply (iii). If $\A(c_1)=0$ and $c_1=i\tau$ where $\tau=2k\pi/\log(m_0/n_0)$
 with $k\in\ZZ^*,$ $m_0,n_0\geq 1$, $m_0\neq n_0$, we may apply (i) (see also Remark \ref{rem:multiplicity}).
\end{proof}

\begin{question}
Does there exist a diagonal operator associated to an injective sequence $(\lambda_n)$
 which fails to have a minimal commutant? If so, can we give a characterization of diagonal operators with a minimal commutant? 
\end{question}

\section{Commutant of composition operators with large characteristic}

\label{sec:c02}

We now handle the case of symbols with characteristic greater than or equal to $2$ by proving Theorem \ref{thm:c02}.
The idea of the proof is to find another composition operator which commutes with $C_\varphi$ and which is not in $\overline{\textrm{alg}(C_\varphi)}^\sigma$.
To find this composition operator, we will start from the K\"{o}nigs map of $\varphi$.
To prove that it is not in $\overline{\textrm{alg}(C_\varphi)}^\sigma$, we shall use the following lemma.

\begin{lemma}\label{lem:notinthealgebra}
Let $\varphi(s)=c_0s + \psi(s)$, $\tilde\varphi(s)=\tilde{c_0}s+\tilde\psi(s)\in\GH$ be such that $c_0,\tilde{c_0}\geq 2$ 
and the characteristic of $\tilde\varphi$ is not a power of the characteristic of $\varphi$. Then $C_{\tilde\varphi}\notin \overline{\mathrm{alg}(C_\varphi)}^\sigma$.
\end{lemma}

\begin{proof}
We proceed by contradiction and we assume
that there is a net $(P_\alpha)\subset\mathbb C[X]$ such that $P_\alpha(C_\varphi)$ converges
to $C_{\tilde\varphi}$ in the weak operator topology. Write each $P_\alpha$ as $P_\alpha(X)=\sum_{k}a_{\alpha,k}X^k.$
Assume first that the characteristic of $\tilde\varphi$ is greater than the characteristic of $\varphi$ and let $\ell\geq 1$ be such that $c_0^\ell<\tilde{c_0}<c_0^{\ell+1}$. 
For $j,m\in\NN$ we set
\begin{align*}
 d_{j,m}&=\langle 2^{-\varphi^{[j]}(s)},2^{-ms}\rangle\\
 e_m&=\langle 2^{-\tilde\varphi(s)},2^{-ms}\rangle.
\end{align*}
Lemma \ref{lem:exponent} tells us that $d_{j,m}=0$ provided $c_0^j>m,$ that $d_{j,c_0^j}\neq 0,$ that $e_m=0$ provided $m<\tilde{c_0},$
and that $e_{\tilde c_0}\neq 0.$ Hence, for $m=c_0^j$ with $j\leq\ell,$ we get
$$\langle P_\alpha(C_\varphi)(2^{-s}),2^{-ms}\rangle=\sum_{k=0}^j a_{\alpha,k}d_{k,c_{0}^j}.$$
Now, this goes to $\langle C_{\tilde\varphi}(2^{-s}),2^{-ms}\rangle$ which is equal to $0$. Taking successively $j=0,1,\dots,\ell,$ we get by a straightforward induction that $a_{\alpha,k}$ tends
to $0$ for all $k$ in $\{0,\dots,\ell\}$. Hence,
$$\langle P_\alpha(C_\varphi)(2^{-s}),2^{-\tilde{c_0}s}\rangle=\sum_{k=1}^\ell a_{\alpha,k}d_{k,\tilde{c_0}}$$
goes to zero. This contradicts that its limit is $\langle C_{\tilde\varphi}(2^{-s}),2^{-\tilde{c_0}s}\rangle=e_{\tilde c_0}\neq 0$. When $\tilde{c_0}<c_0$, the proof is simpler since for all $P\in\CC[X],$ 
$$\langle P(C_\varphi)(2^{-s}),2^{-\tilde{c_0}s}\rangle=0$$
whereas 
$$\langle 2^{-\tilde{\varphi}(s)},2^{-\tilde{c_0}s}\rangle\neq 0.$$
\end{proof}

Endowed with this lemma, we can prove a more general result than Theorem \ref{thm:c02}.
\begin{theorem}\label{thm:c02bis}
Let $\varphi\in\GH$ with $\textrm{char}(\varphi)\geq 2$ and let $u\in\GH$ be its K\"onigs function. Assume that there exists $\sigma_1\geq 0$ such that $u_{|\CC_{\sigma_1}}$ is one-to-one and, for all $c>0$ large enough, $cu(\CC_0)\subset u( \CC_{\sigma_1})$.
Then $C_\varphi$ fails to have a minimal commutant.
\end{theorem}
\begin{proof}
Let $c$ be an integer which is not a power of $c_0$ and satisfies $cu(\CC_0)\subset \CC_{\sigma_1}$.
We define
$$\tilde\varphi(s)=u^{-1}(cu(s))$$
where $u^{-1}$ is the reciprocal of $u_{|\CC_{\sigma_1}}$. Since $cu\in\GD,$ $u^{-1}\in\GD$ we get that $\tilde\varphi\in\GD$. Moreover, $\tilde\varphi$ is defined on $\CC_0$
and $\tilde\varphi(\CC_0)\subset\CC_{\sigma_1}\subset\CC_0$.
In particular, we obtain that $\tilde\varphi\in\GH$ with $\textrm{char}(\tilde\varphi)=c$.
Let now $s\in\CC_0$. On the one hand,
\begin{align*}
u\circ \tilde\varphi\circ\varphi(s)&=cu\circ\varphi(s)\\
&=cc_0 u(s).
\end{align*}
On the other hand,
\begin{align*}
u\circ \varphi\circ\tilde\varphi(s)&=c_0 u(\tilde\varphi(s))\\
&=c_0 c_0 c u(s).
\end{align*}
Since $u_{|\CC_{\sigma_1}}$ is univalent, $\tilde \varphi\circ\varphi=\varphi\circ\tilde\varphi$ on some half-plane $\CC_\sigma$, therefore on $\CC_0$ by analytic continuation.
Hence the maps $\tilde \varphi$ and $\varphi$ commute and we deduce that
the composition operators $C_\varphi$ and $C_{\tilde\varphi}$ also commute.
Then we conclude by Lemma \ref{lem:notinthealgebra}.
\end{proof}

\begin{proof}[Proof of Theorem \ref{thm:c02}]
 Let $u\in\GH$ be the K\"{o}nigs function of $\varphi$.
 There exists $\delta>0$ such that $u(\CC_\veps)\subset\CC_\delta$. From $u\circ \varphi=c_0 u,$
 we then obtain that there exists $\eta>0$ such that $u(\CC_0)\subset\CC_\eta$.
 By Lemma \ref{lem:reverseinclusion}, there exist $\sigma_0,\sigma_1>0$ such that
 $u_{|\CC_{\sigma_1}}$ is one-to-one and $\CC_{\sigma_0}\subset u(\CC_{\sigma_1})$.
 We deduce that $cu(\CC_0)\subset u(\CC_{\sigma_1})$
for all sufficiently large integers $c$. The result now follows from Theorem \ref{thm:c02bis}.
\end{proof}

Theorem \ref{thm:c02bis} is also useful to show that a composition operator with an affine symbol and characteristic greater than or equal to $2$ does not have the minimal commutant property.
\begin{corollary}
Let $\varphi(s)=c_0s+c_1$ with $c_0\geq 2$ and $\A(c_1)\geq 0$. Then $C_\varphi$ fails to have a minimal commutant. 
\end{corollary}
\begin{proof}
In that case,  the K\"{o}nigs map of $\varphi$ is given by $u(s)=s+c_1/(c_0-1)$. It is plain that it satisfies the assumptions of Theorem \ref{thm:c02bis}.
\end{proof}

\begin{remark}
For the case of affine symbols, choosing $c=c_0+1,$ we can easily give a formula for $\tilde\varphi$:
$$\tilde\varphi(s)=(c_0+1)s+\frac{c_0c_1}{c_0-1}.$$
\end{remark}


\section{Commutant of composition operators with characteristic equal to $1$}

\label{sec:c01}

\subsection{Statement of the result} In this section we study composition operators whose symbol has characteristic equal to $1$. We shall prove the following theorem, which goes in the opposite direction to the results of the previous section.

\begin{theorem}\label{thm:characteristic1}
Let $\varphi(s)=s+\psi(s)\in\GH$ be univalent, and $\psi$ is not constant.
Let $u$ be the K\"{o}nigs function
of $\varphi$ and let $\Omega=\{w\in\CC_0:\ w+\overline{u(\CC_0)}\subset u(\CC_0)\}$.
Assume that $u\in\GH$, that $\Omega$ is open and contains $(0,+\infty)$. Then $C_\varphi$ has a minimal commutant.
\end{theorem}

\begin{remark}
The statement does not depend on the map $u$ that we choose, since there are equal up to an additive constant.
Observe also that since $\psi(s)=\sum_{n\geq 1}c_n n^{-s}$ is not constant and maps $\CC_0$ into $\CC_0,$ $\A(c_1)>0$.
\end{remark}

This theorem is inspired by Theorem 6.1 of \cite{llsr18} where a similar theorem is given for composition operators acting on $H^2(\DD)$. Our proof will follow the lines of \cite{llsr18}, even if working with Dirichlet series add substantial difficulties.

Of course, Theorem \ref{thm:characteristic1} is interesting only if we are able to provide maps satisfying its assumptions.
For the case of the disc, the example which was given was that of loxodromic linear fractional maps
(where we can compute explicitly the K\"{o}nigs function).
We do not have such a tool at our disposal and part of our work will be devoted to producing such examples. As a consequence of our methods,
we will get new examples in the case of the disc.

\subsection{Proofs}
In what follows, we fix $\varphi\in\GH$ univalent with characteristic $1$ and we denote by $u$ its K\"{o}nigs map which is also univalent.
We will write $\varphi(s)=s+\sum_{n\geq 1}c_n n^{-s}$ and will always assume that $\A(c_1)>0$.

We first need a complement to Lemma \ref{lem:compogd}.
\begin{lemma}\label{lem:compoh2}
There exists $\nu>0$ such that, for all $f\in\mathcal H^2$, there exists a sequence of complex numbers $(b_n)$ such that, for all $s\in \CC_{\nu}$, $f\circ u^{-1}(s)$ is well defined and equal to $\sum_{n\geq 1}b_n n^{-s}$. Moreover, for all $n\geq 1,$
$$|b_n|\leq \zeta(2)^{1/2} n^\nu \|f\|_2.$$
\end{lemma}
\begin{proof}
 By Lemma \ref{lem:reverseinclusion}, there exists $\nu>0$ such that $\CC_{\nu}\subset u(\CC_{1})$.
 Therefore, $f\circ u^{-1}(s)$ is well defined provided $\A(s)>\nu$. By Lemma \ref{lem:inverse} and Lemma \ref{lem:compogd}, $f\circ u^{-1}$ belongs to $\mathcal D$. Moreover, for any $s\in\CC_\nu,$ 
$$|f\circ u^{-1}(s)|\leq \zeta (2\A(u^{-1}(s)))^{1/2}\|f\|_2\leq \zeta(2)^{1/2}\|f\|_2.$$
Therefore, $f\circ u^{-1}$ is bounded on $\CC_\nu$ so that there exists a sequence of complex numbers $(b_n)$ such that
$f\circ u^{-1}(s)=\sum_{n\geq 1}b_n n^{-s}$
for all $s\in\CC_\nu.$ 
 Finally, for all $n\geq 1$, we obtain
 \begin{align*}
  |b_n n^{-\nu}|&=\left|\lim_{T\to+\infty}\frac1{2T}\int_{-T}^T n^{it} f\circ u^{-1}(\nu+it)dt\right|\\
  &\leq \zeta(2)^{1/2} \|f\|_2.
 \end{align*}
\end{proof}

We fix $\nu>0$ given by the previous lemma (it only depends on $u$).
\begin{lemma}\label{lem:representation}
 Let $f\in\mathcal H^2$ and write, for $\A(s)\geq \nu,$
  $$f\circ u^{-1}(s)=\sum_{n\geq 1}b_n n^{-s}.$$
Let $(\lambda_n)$ be a sequence of complex numbers such that $\sum_n|\lambda_n|n^{\nu}$ converges. Then the series
$$\sum_{n\geq 1}\lambda_n b_n n^{-u}$$
converges in $\mathcal H^2$. More precisely, for all $N\geq 1,$
$$\left\|\sum_{n\geq N}\lambda_n b_n n^{-u}\right\|_2\leq \zeta(2)^{1/2} \|f\|_2 \sum_{n\geq N} |\lambda_n| n^{\nu}.$$
\end{lemma}
\begin{proof}
 It is well known that $C_u$ is a contraction on $\mathcal H^2$ so that, for all $n\geq 1,$ $\|n^{-u}\|_2\leq 1.$
 Therefore the result follows from the estimate on $b_n$ and the triangle inequality.
\end{proof}
Let $\Omega=\{w\in\CC:\ w+\overline{u(\CC_0)}\subset u(\CC_0)\}$. By Lemma \ref{lem:reverseinclusion}, there exists $\sigma_0\geq 0$
such that $\CC_{\sigma_0}\subset \Omega$. For $w\in\Omega$ and $s\in\CC_0,$ we set
$$\varphi_w(s)=u^{-1}(w+u(s)).$$
Lemma \ref{lem:inverse} and \ref{lem:compogd} ensure that $\varphi_w\in \GH$ with characteristic $1$.

\begin{lemma}\label{lem:approximation}
 There exists $\sigma_1>0$ such that, for every $w\in\CC_{\sigma_1}$, for all $T\in\{C_\varphi\}',$ for all $\veps>0,$
 there exists $P\in\CC[X]$ such that
 $$\|P(C_\varphi)-TC_{\varphi_w}\|<\veps.$$
\end{lemma}
\begin{proof}
 Let $T\in\{C_\varphi\}'$ and $\veps>0$. Since any $n^{-c_1}$ is an eigenvalue of $C_\varphi$ of multiplicity $1$ with associated
 eigenvector $n^{-u}$, there exists a sequence of complex numbers $(\lambda_n)_{n\geq 1}$ such that $Tn^{-u}=\lambda_n n^{-u}$
 and $\|(\lambda_n)\|_\infty\leq \|T\|.$ Let $f\in\mathcal H^2$ and write, for $s\in\CC_{\nu},$
 $$f\circ u^{-1}(s)=\sum_{n\geq 1} b_n n^{-s}.$$
 Thus, for $\A(s)>0$ and $\A(w)>\max(\nu,\sigma_0),$
 $$f\circ\varphi_w(s)=\sum_{n\geq 1}b_n n^{-w} n^{-u}.$$
 If moreover $\A(w)\geq\nu+2,$ the right hand side of this equality converges in $\mathcal H^2$. This yields
$$TC_{\varphi_w}f=\sum_{n\geq 1}b_n \lambda_n n^{-w}n^{-u}.$$
 Let $N\geq 1$, $\eta>0$ be such that
 $$\sum_{n\geq N+1}n^{-2}\leq\veps\textrm{ and }\eta\sum_{n=1}^{+\infty}n^{-2}\leq\veps.$$
 Let also $m\in\NN$ be such that $m\A(c_1)>\max(\nu+2,\sigma_0)$ and consider the compact set
 $$K=\{n^{-c_1}:\ n\geq 1\}\cup\{0\}.$$
 Its complement is connected and by Mergelyan's theorem, there exists $Q\in\CC[X]$ such that
 $$
 \left\{
 \begin{array}{rcll}
  \displaystyle \left| Q(n^{-c_1})-\frac{\lambda_n n^{-w}}{n^{mc_1}}\right|&<&\eta&\textrm{ for }n\in\{1,\dots,N\}\\[0.5cm]
  \displaystyle \left| Q(n^{-c_1})\right|&<&1&\textrm{ for }n\geq N+1.
 \end{array}
 \right.
 $$
 We set $P(z)=z^m Q(z)$ and observe that
  $$
 \left\{
 \begin{array}{rcll}
  \displaystyle \left| P(n^{-c_1})-\lambda_n n^{-w}\right|&<&\eta n^{-m\A(c_1)}&\textrm{ for }n\in\{1,\dots,N\}\\
  \displaystyle \left| P(n^{-c_1})\right|&<& n^{-m\A(c_1)}&\textrm{ for }n\geq N+1.
 \end{array}
 \right.
 $$
 For $k\geq m,$ we know that
 $$C_\varphi^k f=f\circ\varphi_{kc_1}=\sum_{n\geq 1}b_n n^{-kc_1}n^{-u}$$
 so that by linearity
 $$P(C_\varphi)f=\sum_{n\geq 1}b_n P(n^{-c_1})n^{-u}.$$
 Therefore
 \begin{align*}
  & \|P(C_\varphi)(f)-TC_{\varphi_w}f\|_2\\
  \leq & \left\|\sum_{n=1}^N b_n \big(P(n^{-c_1})-\lambda_n n^{-w}\big) n^{-u}\right\|_2+\\
  &\quad\quad\left\|\sum_{n=N+1}^{+\infty}b_n P(n^{-c_1})n^{-u}\right\|_2+\left\|\sum_{n=N+1}^{+\infty}b_n\lambda_n n^{-w}n^{-u}\right\|_2\\
 \leq & \zeta(2)^{1/2}\|f\|_{2}\eta\sum_{n=1}^N n^{\nu}n^{-m\A(c_1)}+\\
  &\quad\quad\zeta(2)^{1/2}\|f\|_{2} \sum_{n=N+1}^{+\infty}n^{\nu}n^{-m\A(c_1)}+\zeta(2)^{1/2}\|f\|_{2}\|T\|\sum_{n=N+1}^{+\infty}n^\nu n^{-\A(w)}\\
 \leq &\zeta(2)^{1/2}\|f\|_{2} (2+\|T\|)\veps.
 \end{align*}
Therefore the lemma is proved for $\sigma_1=\max(\sigma_0,\nu)+2$.
\end{proof}

We need a last lemma.
\begin{lemma}
 Let $\Lambda:\mathcal B(\mathcal H^2)\to\CC$ be a linear functional which is continuous with respect to the weak operator topology
 and let $T\in\mathcal B(\mathcal H^2)$.  Assume that $\Omega$ is open. Then $F:\Omega\to\CC,$ $F(w)=\Lambda(TC_{\varphi_{w}})$ is analytic.
\end{lemma}
\begin{proof}
 Any linear functional on $\mathcal B(\mathcal H^2)$ which is continuous with respect to the weak operator topology
 can be expressed as a linear combination of finitely many linear functionals of the form $\langle Rf,g\rangle$,
 $f,g\in\mathcal H^2$ (see \cite[Chapter 2]{Takesaki}). Hence we may assume that $\Lambda(T)=\langle Tf,g\rangle$ so that
 \begin{align*}
  F(w)&=\langle TC_{\varphi_w}f,g\rangle\\
  &=\langle C_{\varphi_w}f,T^*g\rangle.
 \end{align*}
By Lemma \ref{lem:compogd}, we may write for $\A(s)$ large enough and $w\in\Omega,$
$$u^{-1}(u(s)+w)=s+\sum_{n\geq 1}c_n(w)n^{-s}$$
where each map $w\in\Omega \mapsto c_n(w)$, $n\geq 1$, is analytic in $\Omega$.
We do the same thing for $f\circ\varphi_w(s)=f(u^{-1}(u(s)+w))$ which can be written
$\sum_{n\geq 1}a_n(w) n^{-s}$ where each $w\in\Omega\mapsto a_n(w)$, $n\geq 1,$
is holomorphic as a finite linear combination of finite products of $k\mapsto c_k(w)$.
Writing $T^* g=\sum_{n\geq 1}d_n n^{-s}$, we get
$$F(w)=\sum_{n\geq 1}a_n(w)\overline{d_n}.$$
It follows from a result of Mattner (see \cite{Mattner}) that $F$ is holomorphic on $\Omega$ since, for all $w\in\Omega,$
by the Cauchy-Schwarz inequality,
\begin{align*}
 \sum_{n\geq 1}|a_n(w) \overline{d_n}|&\leq \|C_{\varphi_w}(f)\|_2\cdot \|T^* g \|_2\\
 &\leq \|f\|_2 \|T^* g\|_2.
\end{align*}
\end{proof}

We are now ready for the proof of Theorem \ref{thm:characteristic1}.
\begin{proof}[Proof of Theorem \ref{thm:characteristic1}]
 Let $\Lambda$ be a linear functional on $\mathcal B(\mathcal H^2)$ which is continuous in the weak operator topology and that vanishes
 on $\{P(C_\varphi):\ P\in\CC[X]\}$. Let $T\in\{C_\varphi\}'$. We need to prove that $\Lambda(T)=0$. Let us consider $F(w)=\Lambda(TC_{\varphi_w})$
 which is analytic on $\Omega$. By Lemma \ref{lem:approximation}, we know that $F=0$ on some half-plane $\CC_{\sigma_1}$,
 therefore also on $(0,+\infty)$ by analytic continuation. We will get the conclusion if we are able to prove
 that $C_{\varphi_\sigma}\to I$ in the weak operator topology as $\sigma\to 0$. Namely we need to show that
 $\langle C_{\varphi_\sigma}f,g\rangle$ tends to $\langle f,g\rangle$ for every $f,g\in\mathcal H^2$.
 Since the family $(C_{\varphi_\sigma})$ is uniformly bounded, it suffices to verify the convergence when $g$ belongs
 to a total set, for instance $\{k_w:\ \A(w)>1/2\}$ where $k_w$ is the reproducing kernel at $w$. Now, for all $w\in\CC_{1/2},$
 \begin{align*}
  \langle C_{\varphi_\sigma}f,k_w\rangle&=f(\varphi_\sigma(w))\\
  &=f(u^{-1}(\sigma+u(w)))\\
  &\to f(w)=\langle f,k_w\rangle
 \end{align*}
 as $\sigma\to 0.$
\end{proof}

\subsection{Examples}
We need to give examples of maps satisfying the assumptions of Theorem \ref{thm:characteristic1}.
We start with a rather general result.
\begin{proposition}\label{prop:open}
 Let $u:\CC_0\to\CC_0$ be univalent and holomorphic and assume that it extends continuously to $i\RR$. Assume that
 \begin{itemize}
  \item[(a)] $t\mapsto \B(u(it))$ is a bijection of $\RR$;
  \item[(b)] for all $A>0$, there exists $T\in\mathbb R$ such that, for all $\sigma>0$ and for all $t\in\RR$
  \begin{align*}
   t\geq T&\implies \B(u(\sigma+it))\geq A\\
   t\leq -T&\implies \B(u(\sigma+it))\leq -A;
  \end{align*}
  \item[(c)] for all $A>0$, there exists $\sigma_0>0$ such that, for all $t\in\RR,$ for all $\sigma\geq\sigma_0,$
  $$\A(u(\sigma+it))\geq A.$$
 \end{itemize}
 Let $\Omega=\{w\in\CC_0:\ w+\overline{u(\CC_0)}\subset u(\CC_0)\}$. Then $\Omega$ is open and $(0,+\infty)\subset\Omega$.
\end{proposition}

\begin{proof}
 Denote by $\alpha(t)=\A (u(it))$ and by $\beta(t)=\B (u(it))$. It suffices to show that
 $$u(\CC_0)=\{\sigma+i\beta(t):\ t\in\mathbb R,\ \sigma>\alpha(t)\}.$$
Take $w\in u(\CC_0)$ and write it $w=\sigma+i\beta(t)$ for some $\sigma>0$ and some $t\in\mathbb R$. Assume that $\sigma\leq \alpha(t)$.
By the open mapping theorem, we can assume that $\sigma<\alpha(t)$.
Define $\sigma_0=\inf\{\A(\tilde w):\ \B(\tilde w)=\beta(t)\textrm{ and }\tilde w\in u(\CC_0)\}$ so that $\sigma_0<\alpha(t)$. Let $(s_n)\subset\CC_0$ be such that
$u(s_n)\to \sigma_0+i\beta(t)$. Since $|u(s)|\to +\infty$ as $|s|\to+\infty$, we may assume that $(s_n)$ is bounded
and, extracting if necessary, that it converges to some $s\in\overline{\CC_0}$, so that $u(s)=\sigma_0+i\beta(t)$.
If $s\in\CC_0$, this contradicts the definition of $\sigma_0$ (again by applying the open mapping theorem). If $s=i\tau,$ then $\alpha(i\tau)<\alpha(it)$ and $\beta(i\tau)=\beta(it),$ which contradict  the injectivity of $\beta$.
Therefore, $u(\CC_0)\subset \{\sigma+i\beta(t):\ t\in\mathbb R,\ \sigma>\alpha(t)\}.$
Conversely, let $w=\sigma+i\beta(t)$ with $\sigma>\alpha(t)$. Let $T,A>0$ be very large and let $K=[0,A]\times [-T,T]$, $\Gamma=\partial K$.
Then the result above shows that $u$ is injective on $K$ so that, using the assumptions, $u(\Gamma)$ is a simple curve that surrounds $w$. Therefore, $w\in u(\CC_0)$.
\end{proof}

We now show how the K\"{o}nigs map of $\varphi\in\GH$ may satisfy the above assumptions.

\begin{theorem}\label{thm:exs}
 Let $\varphi(s)=s+\psi(s)\in\GH$ be such that $\psi$ extends $\mathcal C^1$ to $i\mathbb R$. Assume that
 \begin{itemize}
  \item[(a)] there exist $p\geq 1$ and $\delta>0$ such that $\varphi^{[p]}(\overline{\CC_0})\subset\CC_\delta.$
  \item[(b)] for each $n\geq 0,$ $|\psi'|$ is bounded on $\varphi^{[n]}(\overline{\CC_0})$ by $M_n>0$.
  \item[(c)] $\displaystyle 1-\left(\prod_{n\geq 0}(1+M_n)\right)\sum_{n=0}^{+\infty}M_n>0$.
 \end{itemize}
Then $C_\varphi$ has a minimal commutant.
\end{theorem}

\begin{proof}
We first observe that the assumptions imply $M_0<1,$ so that $\varphi$ is univalent on $\CC_0$.
 Let $u$ be the K\"{o}nigs map of $\varphi$ which belongs to $\GH$ by (a). We have to prove that $u$ satisfies the assumptions of Proposition \ref{prop:open}.
 We follow the proof of Theorem \ref{thm:lfm} and we use its notations. We know that, for all $s\in\overline{\CC_0}$,
 \begin{equation}\label{eq:thmexs}
 \varphi^{[n]}(s)=s+nc_1+\sum_{j=0}^{n-1}\psi(\varphi^{[j]}(s)).
 \end{equation}
 Assumption (a) implies the uniform convergence of $(u_n)$ on $\overline{\CC_0}$ which implies that $u$ extends continuously to $i\mathbb R$.
 Let $C_n=\sup_{s\in\overline{\CC_0}} |(\varphi^{[n]})'(s)|$ so that $C_0=1$, $C_1=1+M_0$. Then by derivating
 \eqref{eq:thmexs} we get, for all $n\geq 1$,
 $$C_n\leq 1+\sum_{j=0}^{n-1}C_j M_j.$$
 Therefore, if we denote by $(A_n)$ the sequence satisfying $A_0=1$ and
 $$A_n=1+\sum_{j=0}^{n-1}A_j M_j=A_{n-1}+A_{n-1}M_{n-1}=A_{n-1}(1+M_{n-1})$$
 we get $C_n\leq A_n\leq\ell:=\prod_{k\geq 0}(1+M_k)$ for all $n\geq 1$. Let now $\sigma\geq 0$ and let us set, for $t\in\mathbb R,$ $f_n(t)=\B(\varphi^{[n]}(\sigma+it))$. Then
 $$f_n'(t)\geq 1-\sum_{j=0}^{n-1}\ell M_n\geq a:=1-\ell\sum_{j=0}^{+\infty}M_j.$$
 This implies
 $$\B(u_n(\sigma+it_1))-\B(u_n(\sigma+it_0))\geq a(t_1-t_0)\textrm{ for all }t_0<t_1.$$
 Taking the limit yields
 \begin{equation}
 \B(u(\sigma+it_1))-\B(u(\sigma+it_0))\geq a(t_1-t_0) \label{eq:largeimaginarypart}
 \end{equation}
 so that $t\in\RR\mapsto \B(u(it))$ is a bijection from $\RR$ onto $\RR$. Moreover, since $\B(u([0,+\infty))$ is bounded, we
 easily deduce from \eqref{eq:largeimaginarypart} that $u$ satisfies assumption (b) of  Proposition \ref{prop:open}.
 Finally, assumption (c) is satisfied by all maps in $\GH.$
\end{proof}
\begin{remark}
 Under the assumption $\varphi^{[p]}(\overline{\CC_0})\subset\CC_\delta,$
 and using Lemma \ref{lem:dirinfini} and \ref{lem:iteration},
 there exist $C,\sigma>0$ such that, for all $n\in\mathbb N,$
 $M_n\leq C 2^{-n\sigma}$ which ensures that $\sum_{n}M_n$ is convergent.
\end{remark}

We now show a concrete example.
\begin{example}\label{ex:characteristic1}
 Let $\varphi(s)=s+a+b 2^{-s}$ where $b\in\CC^*,$ $a>|b|$ and
 \begin{equation}\label{eq:assumptionexs}
 \frac{|b|\log(2)}{1-2^{-a+|b|}}\exp\left(\frac{|b|\log( 2)}{1-2^{-a+|b|}}\right)<1.
 \end{equation}
 Then $C_\varphi$ has a minimal commutant.
\end{example}
This is for instance the case if $a=5/2$ and $b=1/2$.
\begin{proof}
 It is easy to check that $\varphi^{[n]}(\overline{\CC_0})\subset \CC_{n(a-|b|)}$ for all $n\geq 0$.
 Therefore, with the notations of Theorem \ref{thm:exs},
 $$M_n\leq b \log( 2)2^{-n(a-|b|)}.$$
 We get
 \begin{align*}
  \ell=\prod_{n=0}^{+\infty}(1+M_n)&=\exp\left(\sum_{n=0}^{+\infty}\log(1+M_n)\right)\\
  &\leq \exp\left(\sum_{n=0}^{+\infty} |b|\log( 2) 2^{-n(a-|b|)}\right)\\
  &\leq \exp\left(\frac{|b|\log( 2)}{1-2^{-a+|b|}}\right)
 \end{align*}
and
$$\sum_{n=0}^{+\infty}M_n=\frac{|b|\ln 2}{1-2^{a-|b|}}.$$
The condition $1-\ell\sum_{n=0}^{+\infty}M_n>0$ is clearly implied by \eqref{eq:assumptionexs}.
\end{proof}

\begin{remark}
During the proof of Theorem \ref{thm:characteristic1}, the assumption $u\in\GH$ was used to be sure that $n^{-u}\in\mathcal H^2$ for all $n\geq 1$ with $\|n^{-u}\|\leq 1.$ It turns out that one only needs $\|n^{-u}\|\leq n^{\sigma}$ for some $\sigma>0$ and for all $n\geq 1.$ Therefore, one could replace the assumption ``$u\in\GH$'' by ``there exists $\sigma>0$ such that for all $n\geq 1,$ $n^{-u}\in\mathcal H^2$ with $\|n^{-u}\|\leq n^\sigma$''. We have already observed that provided $C_\varphi$ is compact on $\mathcal H^2,$ then $n^{-u}\in\mathcal H^2$ for all $n\geq 1$. Unfortunately, we do not have a control of the norm of $n^{-u}$.
\end{remark}


\section{Commutant of composition operators on the disc} \label{sec:disc}

In this section, we show that the methods used for Dirichlet series have also applications to
composition operators on the Hardy space $H^2(\DD)$. 
We focus on self maps $\varphi$ of $\DD$ satisfying $\varphi(0)=0$.

\subsection{Zero derivative at the attractive fixed point}

We first concentrate on the case $\varphi'(0)=0$. It seems that under this assumption, the minimal commutant property
was settled only if $\varphi(z)=z^2$ (see \cite{cl98}); for this symbol, $C_\varphi$ does not have the minimal commutant property. 
We intend to give a much more general result, which may be applied to symbols writing $\varphi(z)=\lambda z^p \psi(z)$ with $p\geq 2,$ $\lambda\neq 0,$ $\psi(0)=1$ and $\psi$
does not vanish on $\DD$. For these symbols, B\"ottcher's equation $u\circ\varphi = \lambda u^p$
plays the role of Schr\"oder's equation. It is more difficult to solve on the whole disc. It is known (see \cite[p. 124-127]{Val})
that, provided $\varphi(0)=\cdots=\varphi^{(p-1)}(0)=0$ and $\varphi^{(p)}(0)\neq 0,$ there exists a solution to this equation which is defined, analytic
and univalent in a neighbourhood of $0$.

We shall also need the following analogue of Lemma \ref{lem:notinthealgebra}.
\begin{lemma}\label{lem:notinthealgebradisc}
  Let $\varphi,\tilde \varphi:\DD\to\DD$ be holomorphic and such that there exist $p,q\geq 2$ with $q\notin p\NN$ and
$$\begin{array}{ll}
\varphi^{(j)}(0)=0\textrm{ for }j=1,\dots,p-1&\varphi^{(p)}(0)\neq 0,\\
\tilde \varphi^{(j)}(0)=0\textrm{ for }j=1,\dots,q-1&\tilde\varphi^{(q)}(0)\neq 0,\\
\end{array}$$
then $C_{\tilde\varphi}\notin \overline{\mathrm{alg}(C_\varphi)}^\sigma.$
\end{lemma}
\begin{proof}
  We proceed by contradiction and we assume
that there is a net $(P_\alpha)\subset\mathbb C[X]$ such that $P_\alpha(C_\varphi)$ converges
to $C_{\tilde\varphi}$ in the weak operator topology. Write each $P_\alpha$ as $P_\alpha(X)=\sum_{k}a_{\alpha,k}X^k.$
The key point is that there exists $\lambda,\mu\neq 0$ such that, for all $j\geq 1,$
\begin{align*}
 \varphi^{[j]}(z)&=\lambda^j z^{pj}\psi_j(z)\textrm{ with }\psi_j(0)\neq 0\\
 \tilde\varphi(z)&=\mu z^q \tilde\psi(z)\textrm{ with }\tilde\psi(0)\neq 0.
\end{align*}
Setting
\begin{align*}
 d_{j,m}&=\langle C_{\varphi^{[j]}}(z),z^m\rangle=\langle \varphi^{[j]},z^m\rangle\\
 e_m&=\langle C_{\tilde\varphi}(z),z^m\rangle=\langle\tilde\varphi,z^m\rangle,
\end{align*}
we get $d_{j,m}=0$ if $pj>m$, $e_m=0$ if $q>m$ and $e_q\neq 0.$ Let $\ell\geq 1$ be such that $\ell p<q<(\ell+1)p$. From
$$\langle P_\alpha(C_\varphi)(z),z^{jp}\rangle=\sum_{k=1}^j a_{\alpha,k}d_{k,jp}\to \langle C_{\tilde \varphi}(z),z^{jp}\rangle=0,$$
for all $j=1,\dots,\ell,$ we get that $a_{\alpha,k}\to 0$ for all $k=1,\dots,\ell$ and this contradicts that
$$\langle P_\alpha(C_\varphi)(z),z^q\rangle=\sum_{k=1}^{\ell}a_{\alpha,k}d_{k,q}\to \langle C_{\tilde \varphi}(z),z^{q}\rangle\neq 0.$$
Therefore, $C_{\tilde \varphi}\notin \overline{\textrm{alg}(C_\varphi)}^\sigma.$
\end{proof}

We are now in position to give a general result.
\begin{theorem}\label{thm:zeroderivative}
 Let $\varphi:\DD\to\DD$ be analytic which writes $\varphi(z)=\lambda z^p \psi(z)$ with $p\geq 2,$ $\lambda\neq 0$ and $\psi(0)=1$.
 Assume that there exists $u:\DD\to\CC$ an analytic solution to B\"ottcher's equation $u\circ\varphi=\lambda u^p$, $u(0)=0,$
 satisfying the two following properties:
 \begin{itemize}
  \item[(a)] there exists $r>0$ such that $u_{|D(0,r)}$ is univalent.
  \item[(b)] there exists $q\in\NN\backslash p\NN$ and $\mu\in \CC$ with $\mu\lambda^q=\lambda \mu^p$ and $\mu u^q(\DD)\subset u(D(0,r)).$
 \end{itemize}
 Then $C_\varphi$ fails to have a minimal commutant. 
\end{theorem}
\begin{proof}
 For $z\in\DD,$ let us set $\tilde\varphi(z)=u^{-1}(\mu u^q(z))$. This defines an analytic self-map of $\DD.$ Moreover, for any $z\in \DD,$
 \begin{align*}
  u\circ\varphi\circ\tilde\varphi(z)&=\lambda u^p(\tilde\varphi(z))\\
  &=\lambda \mu^p u^{pq}(z)
 \end{align*}
whereas 
\begin{align*}
 u\circ\tilde\varphi(z)&=\mu u^q(\varphi(z))\\
 &=\mu \lambda^q u^{pq}(z).
\end{align*}
Since $u_{|D(0,r)}$ is univalent, we get that $\varphi\circ\tilde\varphi=\tilde\varphi\circ\varphi$ in a neighbourhood of $0,$
hence in $\DD.$ Moreover, since $u'(0)\neq 0$ (by univalence of $u_{|D(0,r)}$), it is easy to check that $\varphi$
and $\tilde\varphi$ satisfy the assumptions of Lemma \ref{lem:notinthealgebradisc}. Hence
$C_{\tilde\varphi}\notin \overline{\mathrm{alg}(C_\varphi)}^\sigma.$
\end{proof}

The function $\varphi(z)=z^p$ satisfies the assumptions of the above theorem, since we can take $u(z)=z$ and  $r=1$.

We now give a corollary which will be easier to apply.

\begin{corollary}\label{cor:zeroderivative}
 Let $\varphi:\DD\to\DD$ be analytic which writes $\varphi(z)=\lambda z^p \psi(z)$ with $p\geq 2,$ $\lambda\neq 0$ and $\psi(0)=1$.
 Assume that there exists $u:\DD\to\CC$ an analytic solution to B\"ottcher's equation $u\circ\varphi=\lambda u^p$
 with $u(0)=0,$ $u'(0)\neq 0$ and $|\lambda|\cdot \|u\|_\infty^{p-1}<1.$
 Then $C_\varphi$ fails to have a minimal commutant. 
\end{corollary}
\begin{proof}
 Since $u'(0)\neq 0,$ there exists $r>0$ such that $u_{|D(0,r)}$ is univalent. By the open mapping theorem,
 $u(D(0,r))$ contains some neighbourhood of $0$. Now, the sequence $(\lambda^{(q-1)/(p-1)}\|u\|_\infty^q)_{q\geq 1}$
 goes to zero. Therefore, there exists $q\in\NN\backslash p\NN$ such that, setting
 $\mu=\lambda^{(q-1)/(p-1)},$ 
 $$\mu u^q(\DD)\subset u(D(0,r)).$$
 The result now follows from Theorem \ref{thm:zeroderivative}.
\end{proof}

To get concrete examples, we now need to find a solution to B\"ottcher's equation with adequate estimates. This will
be done thanks to the following lemma.

\begin{lemma}\label{lem:bottcher}
 Let $\varphi:\DD\to\DD$ be holomorphic which writes $\varphi(z)=\lambda z^p \psi(z)$ with $p\geq 2,$ $\lambda\neq 0$, $\psi(0)=1$
 and $\psi$ does not vanish on $\DD.$ 
 Then there exists $u\in H(\DD)$ such that $u\circ\varphi=\lambda u^p,$ $u(0)=0$, $u'(0)=1$ and $ |\lambda|\cdot \|u\|_\infty\leq 1$.
Moreover, if $\|\varphi^{[n]}\|_\infty<1$ for some $n\geq 1,$ then $ |\lambda|\cdot \|u\|_\infty^{p-1}<1.$
\end{lemma}

\begin{proof}
 For $n\geq 1$ and $z\in\DD,$ let us set
 $$u_n(z)=z\psi^{\frac 1p}(z)\cdots \psi^{\frac1{p^n}}(\varphi^{[n-1]}(z)).$$
 We first observe that $(u_n)$ converges uniformly on the compact subsets of $\DD.$
 Indeed, let $K$ be a compact subset of $\DD$ and let us consider $L:\DD\to\CC$ a logarithm of $\psi$.
 Then, for all $n\geq 1$ and all $z\in K,$
 \begin{align*}
  \left|\psi^{\frac1{p^n}}(\varphi^{[n-1]}(z))-1\right|&=\left|\exp\left(\frac 1{p^n}L(\varphi^{[n-1]}(z))\right)-1\right|\\
  &\leq \frac{C_K}{p^n}
 \end{align*}
 since $\{L(\varphi^{[n-1]}(z)):\ n\geq 1,\ z\in K\}$ is a compact subset of $\CC$ ($0$ is the Denjoy-Wolff point of $\varphi$).
 By convergence of $\sum_{n\geq 1}1/p^n,$ the sequence $(u_n)$ converges uniformly on all compact subsets of $\DD$ to a function $u\in H(\DD)$.
 Moreover, for $n\geq 1$ and $z\in \DD,$ writing
 \begin{align*}
  u_n\circ\varphi(z)&=\varphi(z)\psi^{\frac 1p}(\varphi(z))\cdots \psi^{\frac1{p^n}}(\varphi^{[n]}(z))\\
  &=\frac{\varphi(z)}{z^p\psi(z)} z^p \psi(z) \psi^{\frac 1p}(\varphi(z))\cdots \psi^{\frac1{p^n}}(\varphi^{[n]}(z))\\
  &=\lambda u_{n+1}^p(z)
 \end{align*}
 and taking the limit, we get $u\circ\varphi=\lambda u^p$. 
 By definition of $u_n,$ $u_n(0)=0$ and since
 $$u_n'(0)=\prod_{j=1}^n \psi^{\frac 1{p^j}}(0)=1$$
  for all $n\in\NN,$
 we also get $u'(0)=1,$ $u$ is not constant. Finally the definition of $u$ yields
 \begin{align*}
\|u\|_\infty&\leq\prod_{j=1}^{+\infty}\|\psi\|_\infty^{\frac 1{p^j}}=\|\psi\|_{\infty}^{1/(p-1)}.
\end{align*}
We conclude because $|\lambda|\cdot \|u\|_\infty^{p-1}\leq |\lambda|\cdot  \|\psi\|_\infty=\|\varphi\|_\infty.$
If $\|\varphi^{[n]}\|_\infty<1$, then $\|\psi^{\frac 1{p^n}}\circ\varphi^{[n]}\|<\|\psi\|_\infty^{\frac 1{p^n}},$
yielding the strict inequality.
\end{proof}

As a consequence of Corollary \ref{cor:zeroderivative} and Lemma \ref{lem:bottcher}, we obtain the following sufficient
condition which applies to plenty of examples.

\begin{corollary}\label{cor:zeroderivativefinal}
 Let $\varphi:\DD\to\DD$ be holomorphic which writes $\varphi(z)=\lambda z^p \psi(z)$ with $p\geq 2,$ $\lambda\neq 0$, $\psi(0)=1$
 and $\psi$ does not vanish on $\DD.$ Assume that there exists $n\geq 1$ such that $\|\varphi^{[n]}\|_\infty<1.$
 Then $C_\varphi$ does not have a minimal commutant.
\end{corollary}

\begin{example}
 Let $\varphi(z)=\lambda z^2(1+\delta z^2)$ with $\lambda(1+\delta)<1$. Then $C_\varphi$ fails to have a minimal commutant.
\end{example}

\begin{question}
 Can we dispend with the assumption that $\psi$ does not vanish on $\DD$ in Corollary \ref{cor:zeroderivativefinal}?
 Without this assumption, one can still build a solution to B\"ottcher's equation by considering now 
 $$u_n(z)=\varphi^{[k]}(z)\psi^{\frac 1p}(\varphi^{[k+1]}z)\cdots \psi^{\frac1{p^n}}(\varphi^{[k+n]}(z)).$$
 But the limit will not be univalent near the origin.
\end{question}

\subsection{Non-zero derivative at the attractive fixed point}
We now study the minimal commutant property for self-maps $\varphi$ of $\DD$ with $\varphi(0)=0$ et $\varphi'(0)\neq 0$. 
When $\varphi$ is a rotation, $\varphi(z)=\omega z$, it is known that $C_\varphi$ has a minimal commutant if and only
if $\omega$ is not a root of unity (see \cite{llsr18}). 
When $\varphi$ is not a rotation, the following theorem, which inspired Theorem \ref{thm:characteristic1}, is proved in \cite{llsr18}.
Recall that a set $S\subset\CC$ is said to be strictly starlike with respect to the origin if $r\bar S\subset S$ for all $0\leq r<1$. 

\begin{theoremllsr}
Let $\varphi:\DD\to\DD$ be holomorphic, univalent, with $\varphi(0)=0$. Let $u$ be
the K\"{o}nigs function of $\varphi$ and suppose that $u(\DD)$ is bounded and strictly starlike with respect to the origin. Then the operator $C_\varphi$ has a minimal commutant.
\end{theoremllsr}

This theorem is applied in \cite{llsr18} only when $\varphi$ is a linear fractional loxodromic self-map of the unit disc. In that case, one can compute explicitly $u$ and verify that $u(\DD)$ is an open disc containing $0,$ hence it is bounded and strictly starlike with respect to the origin.

Our aim, in this subsection, is to give a sufficient condition on $\varphi$ which allows us to prove that $u(\DD)$ is bounded and strictly starlike with respect to the origin, without computing $u$.
We did not try to get the best possible conditions.
Our aim was to obtain composition operators with minimal commutant which are not induced by a linear fractional map.

\begin{theorem}
Let $\psi:\DD\to\CC$ be holomorphic with extends $\mathcal C^1$ to $\overline \DD$ and $\psi(0)=0$.
Let $\lambda\in\DD$ and let us set $\varphi(z)=\lambda z(1+\psi(z))$.
Assume that $\varphi$ is univalent on $\overline{\DD}$ and that $\varphi(\overline \DD)\subset \DD$. Assume also that, for all $z\in\overline\DD,$
\begin{equation}\label{eq:series}\sum_{j\geq 0}\A\left(\frac{(\varphi^{[j]})'(z) \psi'(\varphi^{[j]}(z))}{1+\psi(\varphi^{[j]}(z))}\right)>-1.
\end{equation}
Then $C_\varphi$ has a minimal commutant.
\end{theorem}

\begin{proof}
The assumptions imply that $0$ is the Denjoy-Wolff point of $\varphi$ with $\varphi'(0)=\lambda$. For $n\geq 1$ and $z\in\overline\DD$, we set
$$u_n(z)=\frac{\varphi^{[n]}(z)}{\lambda^n}.$$
The proof of \cite{sh} ensures that $(u_n)$ converges uniformly on compact subsets of $\DD$ to $u$ and since $\varphi(\overline\DD)$ is a compact subset of $\DD$, we have in fact uniform convergence of $(u_n)$ on $\overline\DD$ and we can extend continuously $u$ to $\overline\DD$. Hence, we may write
$$\varphi^{[j]}(z)=\lambda^j u(z)+\lambda^j \veps_j(z)$$
where $\sup_{j\geq 1,z\in\overline\DD}|\veps_j(z)|<+\infty$. In particular,
$\|\varphi^{[j]}\|_\infty\leq C|\lambda|^j.$ Write now
$$u_n'(z)=\prod_{j=0}^{n-1}\frac{\varphi'(\varphi^{[j]}(z))}{\lambda}$$
and observe that
$$\left|\frac{\varphi'(\varphi^{[j]}(z))}{\lambda}-1\right|\leq C_1 |\varphi^{[j]}(z)|\leq C_2 |\lambda|^j.$$
Hence
we get that the product converges uniformly on $\overline\DD$ so that $u$ is $\mathcal C^1$ on $\overline\DD.$

Furthermore, it is easy to check by induction that, for all $n\geq 1$ and all $z\in\overline\DD,$
$$\varphi^{[n]}(z)=\lambda^n z(1+\psi(z))\cdots (1+\psi(\varphi^{[n-1]}(z)))$$
so that
$$\frac{zu_n'(z)}{u_n(z)}=1+\sum_{j=0}^{n-1}\frac{(\varphi^{[j]})'(z) \psi'(\varphi^{[j]}(z))}{1+\psi(\varphi^{[j]}(z))}.$$
We take the limit as $n\to+\infty$: by \eqref{eq:series} we obtain that,  for all $z\in\overline{\DD},$
$$\A\left(\frac{zu'(z)}{u(z)}\right)>0.$$
Now, fix $\theta_0\in\mathbb R$ and consider $\log$ a determination of the logarithm around $u(e^{i\theta_0})$ and $\arg$ the associated determination of the argument. Then
\begin{align*}
\frac{\partial (\arg(u\circ e^{i\theta}))}{\partial \theta}({\theta_0})&=\frac{\partial\big(\B(\log(u\circ e^{i\theta}))\big)}{\partial\theta}({\theta_0})\\
&=\B \frac{\partial\big(\log(u\circ e^{i\theta})\big)}{\partial\theta}({\theta_0})\\
&=\B \left(ie^{i\theta_0}\frac{u'(e^{i\theta_0})}{u(e^{i\theta_0})}\right)\\
&=\A \left(e^{i\theta_0}\frac{u'(e^{i\theta_0})}{u(e^{i\theta_0})}\right)>0.
\end{align*}
Therefore, the argument of $u(e^{i\theta})$ is locally increasing. Moreover, by the argument principle, since $u$ has only a simple zero at $0$ in $\overline\DD$, the winding number of $u(\TT)$ around the origin is equal to $1$. This easily implies that there exists $g:[0,2\pi]\to[0,2\pi]$ increasing and verifying $g(0)=0$, $g(2\pi)=2\pi$ and $\rho:[0,2\pi]\to(0,+\infty)$ such that
$$u(\overline \DD)=\{\rho e^{ig(\theta)}:\ \theta\in[0,2\pi],\ \rho\in[0,\rho(\theta)]\}.$$
Thus $u(\DD)$ is bounded and strictly starlike with respect to the origin and by Theorem A, $C_\varphi$ has a minimal commutant.
\end{proof}

We deduce the following more readable corollary, where we replace \eqref{eq:series} by a condition which is easier to testify.

\begin{corollary}
Let $\psi:\DD\to\CC$ be holomorphic with extends $\mathcal C^1$ to $\overline \DD$ and $\psi(0)=0$.
Let $\lambda\in\DD$ and let us set $\varphi(z)=\lambda z(1+\psi(z))$. Assume that $\varphi$ is univalent on $\overline{\DD}$ and that $\varphi(\overline \DD)\subset \DD$. Assume also that
$$\left\|\frac{\psi'}{1+\psi}\right\|_\infty\times\frac{1}{1-\|\varphi'\|_\infty}<1.$$
Then $C_\varphi$ has a minimal commutant.
\end{corollary}
\begin{proof}
We observe that, for all $z\in\overline\DD$ and all $j\geq 0,$
$$|(\varphi^{[j]})'(z)|=\left|\prod_{k=0}^{j-1}\varphi'(\varphi^{[k]}(z))\right|\leq \|\varphi'\|_{\infty}^j.$$
Hence, \eqref{eq:series} is satisfied since
\begin{align*}
\sum_{j\geq 0}\A\left(\frac{(\varphi^{[j]})'(z) \psi'(\varphi^{[j]}(z))}{1+\psi(\varphi^{[j]}(z))}\right)\geq -\sum_{j\geq 0} \|(\varphi^{[j]})'\|_\infty \times  \left\|\frac{\psi'}{1+\psi}\right\|_\infty.
\end{align*}
\end{proof}

\begin{example}
Let $\varphi(z)=\lambda z(1+az)$ with $\lambda,a>0$ and
$$\frac{\lambda a(1+2a)}{1-a}<1.$$
Then $C_\varphi$ has a minimal commutant.
\end{example}
\begin{proof}
It is easy to check that $\|\varphi'\|_\infty=\lambda(1+2a)$ so that, setting
$\psi(z)=az$,
$$\left\|\frac{\psi'}{1+\psi}\right\|_\infty=\frac{a}{1-a}.$$
\end{proof}
Thus, if $\varphi(z)=\frac{1}{3}z\left(1+\frac 12z\right)$, then $C_\varphi$ has a minimal commutant.

\begin{question}
Beyond linear fractional maps, there is another family of maps with simple K\"{o}nigs functions, the lens maps.
If we denote by $\mathcal C(z)=(1+z)/(1-z)$ the Cayley map which sends $\DD$ onto $\CC_0$ and if $\alpha\in(0,\pi/2)$,
the lens map $\varphi_\alpha$ is defined by
$$\varphi_\alpha(z)=\frac{\mathcal C(z)^\alpha-1}{\mathcal C(z)^\alpha+1}.$$
Its K\"{o}nigs function is then the function
$$u(z)=\log\left(\frac{1+z}{1-z}\right)$$
and $u(\DD)$ is the strip $\mathbb R\times (-\pi/2,\pi/2)$ which is strictly starlike with respect
to the origin but is unbounded. Does $C_{\varphi_\alpha}$ have a minimal commutant?
\end{question}

\begin{remark}
In view of the above question, it could be observed that one can slightly relax the assumptions of Threorem A. Indeed, the condition $u(\DD)$ bounded is very strong (and equivalent to saying that $\|\varphi^{[n]}\|_\infty<1$ for some positive integer $n$. A careful look at the proof of Theorem A shows that we could replace it by the two following conditions:
\begin{itemize}
\item $\Omega=\{w\in\CC:\ w+\overline{u(\DD)}\subset\DD\}$ is open.
\item There exists $C>0$ such that, for all $n\geq 1,$ $u^n$ belongs to $H^2(\DD)$ with $\|u^n\|\leq C^n.$
\end{itemize}
Unfortunately, for the lens maps above, it is not hard to show that $\|u^n\|/C^n\to+\infty$ for all $C>0.$
\end{remark}

\section{Double commutant property}

This last section is devoted to studying the double commutant property
for composition operators on $\mathcal H^2$ induced by affine symbols
(as this is done on $H^2(\DD)$ for composition operators induced by linear fractional maps in \cite{llsr19}).
Given any subset $\mathcal{A}\in \mathcal{B}(H)$,
recall that the commutant of $\mathcal{A}$ is defined as the family of all operators that commute with all the operators in $\mathcal{A}$,
that is,
\begin{equation*}
  \mathcal{A}':=\bigcap_{A\in\mathcal{A}}\{A\}'=\bigcap_{A\in\mathcal{A}}\{T\in\mathcal{B}(H):TA=AT\}.
\end{equation*}
The double commutant of $\mathcal{A}$ is defined as $\mathcal{A}''=(\mathcal{A}')'$.
Clearly, $\mathcal{A}$ is contained in $\mathcal{A}''$,
and then so is $\overline{\mathcal{A}}^{\sigma}$,
the closure of $\mathcal{A}$ in the weak operator topology $\sigma$.
As is well-known,
$\overline{\mathcal{A}}^{\sigma}=\mathcal{A}''$ holds for any self-adjoint, unital subalgebra $\mathcal{A}\subseteq \mathcal{B}(H)$,
by the von Neumann's double commutant theorem.
A subalgebra $\mathcal{A}\subseteq \mathcal{B}(H)$ is said to possess the double commutant property provided that $\overline{\mathcal{A}}^{\sigma}=\mathcal{A}''$.
Since for $A\in\mathcal B(H),$ $\overline{\textrm{alg}(A)}^\sigma\subset\{A\}''\subset \{A\}'$,
it is clear that an operator with a minimal commutant has the double commutant property. Thus for $\varphi\in \GH$ and its K\"{o}nigs function satisfying the conditions of Theorem \ref{thm:characteristic1}, $C_\varphi$ has the double commutant property.

\subsection{Symbols with characteristic $1$}

In this subsection, we shall prove that any symbol $\varphi(s)=s+c_1$ gives rise to an operator with the double commutant property.
One could think that it is easy since $C_\varphi$ is a diagonal operator and in particular a normal operator.
Therefore, by a result of Turner \cite{Turner}, it has the double commutant property if and only if
every invariant subspace of $C_\varphi$ is a reducing subspace. Nevertheless, it is not so easy to determine
which sequences have this property, see e.g. \cite[Lemma 6]{MarSeu}.

This motivates us to state a sufficient condition similar to Theorem \ref{thm:diagonal} but now for the double commutant property.

\begin{theorem}\label{thm:diagonalbi}
 Let $H$ be a Hilbert space with orthonormal basis $(e_n)$, let $I\subset\mathbb N,$ let $(\lambda_n)_{n\in I}\subset\CC^I$
 be bounded and let $D\in\mathcal B(H)$
 be a diagonal operator with respect to $(e_n)$ such that $\sigma_p(D)=\{\lambda_n:\ n\in I\}$.
 Assume that for all $(\mu_n)_{n\in I}\subset\CC^I$ bounded, there exists $M>0$ such that, for all $N\geq 1,$
 for all $\veps>0,$ there exists $P\in\CC[X]$ such that
  $$\left\{
 \begin{array}{rcl}
  |P(\lambda_n)-\mu_n|&\leq&\veps,\ n\in I\cap\{1,\dots,N\}\\
  \sup_{n\in I}|P(\lambda_n)|&\leq&M.
 \end{array}\right.$$
Then $D$ has the double commutant property.
\end{theorem}
Observe that the assumptions imply that $\lambda_n\neq \lambda_m$ provided $n\neq m,$
but they do not say anything on the dimension of $\ker(C_\varphi-\lambda_n I)$.
\begin{proof}
 For $n\in I,$ let $H_n=\ker(D-\lambda_n I)$ and observe that we have the orthogonal decomposition $H=\bigoplus_n H_n$
 and that $D=\bigoplus_n \lambda_n I_n$ where $I_n$ is the identity map on $H_n$.
 It is easy to check that for all $(A_n)_{n\in I}$ with $A_n\in\mathcal B(H_n),$
 $T=\bigoplus_n A_n$ belongs to $\{D\}'$. This yields that $S$ belongs to $\{D\}''$ if and only if there exists
 a bounded sequence $(\mu_n)_{n\in I}$ of complex numbers such that $S=\bigoplus_n \mu_n I_n$. We
 conclude as in the proof of Theorem \ref{thm:diagonal} that $S\in\overline{\textrm{alg}(D)}^\sigma$.
\end{proof}

\begin{corollary}\label{cor:diagonalbi}
 Let $H$ be a Hilbert space with orthonormal basis $(e_n)$, let $D\in\mathcal B(H)$
 be a diagonal operator with respect to $(e_n)$ which is an isometry. Then $D$ has the double commutant property.
\end{corollary}
\begin{proof}
 This follows from Theorem \ref{thm:diagonalbi} and the Rudin-Carleson interpolation theorem.
\end{proof}

\begin{corollary}
 Let $\varphi(s)=s+c_1$ with $\A(c_1)\geq 0$. Then $C_\varphi$ has the double commutant property.
\end{corollary}
\begin{proof}
 We just have to consider the case $\varphi(s)=s+i\tau,$ $\tau\in\mathbb R,$ which is settled by Corollary \ref{cor:diagonalbi}.
\end{proof}

\subsection{Symbols with large characteristic}

The aim of this subsection is to prove the following result.
\begin{theorem}\label{thm:doublecommutantlarge}
Let $\varphi(s)=c_0 s+c_1$ with $c_0\geq 2$ and $\A(c_1)\geq0$.
 Then $C_\varphi$ has the double commutant property.
\end{theorem}

We first observe that if $\varphi(s)=c_0s+i\tau,$ with $\tau\in\mathbb R,$ then $C_\varphi$ is a non-unitary isometry. Therefore it has the double commutant property
by \cite[Theorem 3.3]{Turner}.
Hence, in the following, we will always consider a symbol $\varphi(s)=c_0s +c_1$ with $\A(c_1)>0$. We first need a general result on stability of the double commutant property.

\begin{lemma}\label{lem:dcsum}
 Let $X$ be a Banach space, let $T\in\mathcal L(X)$ with $\|T\|<1$ and $T$ has the double commutant property.
 Then $1\oplus T$ acting on $\CC\oplus X$ has the double commutant property.
\end{lemma}
\begin{proof}
 We first show that $\{1\oplus T\}'=\{\lambda\oplus B:\ B\in\{T\}'\}$. Let $S\in\{1\oplus T\}'$ and write it
 $$S=\begin{pmatrix}
      \lambda & A\\u&B
     \end{pmatrix}
     $$
so that
$$S(1\oplus T)=\begin{pmatrix}
                \lambda & AT\\  u & BT
               \end{pmatrix}\textrm{ and }
(1\oplus T)S=\begin{pmatrix}
              \lambda &A \\ Tu&TB
             \end{pmatrix}.
$$
If $A$ is not equal to $0$, one may find $x\in X$ with $\|x\|=1$ such that $\|Ax\|>\|A\| \cdot \|T\|$ which contradicts
$$\|Ax\|=\|ATx\|\leq \|A\|\cdot \|T\|.$$
Since $\|T\|<1$, the equality $Tu=u$ also yields $u=0$.

On the other hand, any $S\in\{1\oplus T\}''$ should also write $\lambda \oplus B$ because $\{1\oplus T\}''\subset \{1\oplus T\}'$.
It is also straightforward to check that $B$ belongs to $\{T\}''$.
Therefore, $\{1\oplus T\}''=\CC\oplus \{T\}''$.
Finally, Proposition 2.4 of \cite{llsr19} ensures that
$$\overline{\textrm{alg}(1\oplus T)}^{\sigma}=\CC\oplus \overline{\textrm{alg}(T)}^\sigma=\CC\oplus \{T\}''$$
since $T$ has the double commutant property.
\end{proof}

Coming back to composition operators on $\mathcal H^2$ with symbol $\varphi(s)=c_0 s+c_1$ with $\A(c_1)>0$,
one may decompose $\mathcal H^2=\CC\cdot 1\oplus \mathcal H^2_0$ where $\mathcal H^2_0=\overline{\textrm{span}}(n^{-s}:\ n\geq 2)$
and $\CC\cdot 1$ and $\mathcal H_0^2$ are preserved by $C_\varphi$. Moreover, $\|{C_{\varphi}}_{|\mathcal H^2_0}\|<1$.
Therefore, it is enough to prove that ${C_{\varphi}}_{|\mathcal H^2_0}$
has the double commutant property. In the remaining part of this subsection, we shall only discuss this operator
and $C_\varphi$ will always stand for ${C_{\varphi}}_{|\mathcal H^2_0}$. The key point in our argument is that
this operator can be seen as a direct sum of forward shifts.

Let $w$ be a weight sequence (namely a bounded sequence of nonzero complex numbers) and let $H$ be a Hilbert space with orthonormal
basis $(e_k)_{k\geq 0}$. The forward shift $S_w$ is defined on $H$ by $S_w(e_k)=w_ke_{k+1}$.

Let now $I=\{m\geq 2:\ m\neq n^{c_0^k},\ n\geq 2,\ k>0\}$ and for $m\in I,$ let
$H_m=\textrm{span}(m^{-c_0^k s}:\ k\geq 0)$. Then $\mathcal H^2_0=\bigoplus_{m\in I}H_m$.
Moreover, since $C_\varphi(m^{-c_0^k s})=m^{-c_0^k c_1}m^{-c_0^{k+1}s}$, $H_m$ is a reducing subspace of $C_\varphi$
and $C_\varphi=\bigoplus_{m\in I} S_{w^{(m)}}$ where for $k\geq 0,$ $w_k^{(m)}=m^{-c_0^k c_1}.$

Since we are looking for the commutant of $C_\varphi,$ we shall need to study operators intertwining weighted shifts.
For $w,w'$ two weight sequences and $a=(a_n)\in\CC^\NN,$ we consider the (possibly unbounded) operator
$S_{w,w',a}$ defined on $H$ with orthonormal basis $(e_n)_{n\in\NN}$ by
$$\langle S_{w,w',a}e_j,e_i\rangle=
\left\{
\begin{array}{ll}
 0&\textrm{ provided }j>i\\
 \displaystyle a_{i-j}\frac{w_0\cdots w_{i-1}}{w'_0\cdots w'_{j-1}}&\textrm{ provided }j\leq i.
\end{array}\right.
$$
The following lemma is reminiscent from the work of Shields and Wallen \cite{SW}.
\begin{lemma}\label{lem:intertwiningshifts}
 Let $w,w'$ be two weight sequences and let $T\in\mathcal B(H)$ be such that $TS_{w'}=S_w T$. Then
 there exists $a\in\mathbb C^\NN$ such that $T=S_{w,w',a}$.
\end{lemma}
\begin{proof}
For $i,j\in\mathbb N$, let $t_{i,j}=\langle Te_j,e_i\rangle$ so that, for all $j\in\NN,$ $Te_j=\sum_{i\in\NN}t_{i,j}e_i.$
We thus have
\begin{equation*}
  TS_{w'}e_j=T(w'_j e_{j+1})=\sum_{i=0}^{+\infty}w'_j t_{i,j+1}e_i
\end{equation*}
whereas
$$S_w Te_j=\sum_{i=0}^{+\infty}w_i t_{i,j}e_{i+1}.$$
This yields $t_{0,j}=0$ for all $j\geq 1$ and for all $i,j\geq 0,$
$$t_{i+1,j+1}=\frac{w_i}{w'_j}t_{i,j}.$$
By an immediate induction, $t_{i,j}=0$ for $j>i$ and
$$t_{i,j}=\frac{w_{i-j}\cdots w_{i-1}}{w'_0\cdots w'_{j-1}}t_{i-j,0}$$
provided $j\leq i$. We get the result if we set, for $n\in\NN,$
$$a_n=\frac{1}{w_0\cdots w_{n-1}}t_{n,0}.$$
\end{proof}

We come back to the decomposition of $\mathcal H^2_0$ as $\bigoplus_{m\in I}H_m$.
We denote by $P_m$ the orthogonal projection onto $H_m$ and, for $T\in\mathcal B(\mathcal H^2_0)$, we set
$T_{m,n}=P_m TP_n$, $m,n\in I$. Lemma \ref{lem:intertwiningshifts} allows us to compute the commutant of $C_\varphi$.

\begin{lemma}\label{lem:commutant}
 Let $T\in\mathcal B(\mathcal  H^2_0)$.
 Then $T\in\{C_\varphi\}'$ if and only if
 for each $(m,n)\in I^2,$ there exists a sequence $a^{(m,n)}\in\CC^\NN$ such that,
 $$T_{m,n}=S_{w^{(m)},w^{(n)},a^{(m,n)}}.$$
\end{lemma}

\begin{proof}
We start with some $S\in\{C_\varphi\}''$ and we first show that $S$ is diagonal in $\bigoplus_{m\in I}H_m.$
 Let $T\in\mathcal B(\mathcal H^2_0)$. Then $T\in\{C_\varphi\}'$ if and only if, for all $(m,n)\in I^2,$
 $$P_m T C_\varphi P_n=P_m C_\varphi T P_n.$$
 Since each $H_k$ is a reducing subspace of $C_\varphi,$ this is also equivalent to
 $$P_m T P_n P_n C_\varphi P_n=P_m C_\varphi P_m P_m T P_n$$
 namely
 $$T_{m,n}S_{w^{(n)}}=S_{w^{(m)}} T_{m,n}.$$
 The result now follows immediately from Lemma \ref{lem:intertwiningshifts}.
\end{proof}

With this lemma at hand, we are able to compute the double commutant of $C_\varphi$.
\begin{lemma}\label{lem:doublecommutant}
 Let $S\in\mathcal B(\mathcal H^2_0)$. Then $S\in\{C_\varphi\}''$ if and only if there exists $a\in\CC^\NN$ such that
 $S=\bigoplus_{m\in I}S_{w^{(m)},w^{(m)},a}$.
\end{lemma}
\begin{proof}
 Let $m,n\in I$ with $m<n$. We choose $T\in\{C_\varphi\}'$ such that $T_{k,l}=0$ provided $k\neq l$, $T_{k,k}=0$
 provided $k\notin \{m,n\}$ and $T_{m,m}=S_{w^{(m)},w^{(m)},b}$, $T_{n,n}=S_{w^{(n)},w^{(n)},c}$
 with $b,c\in\CC^\NN$ such that these operators are bounded (this holds true for instance if $b$ and $c$ have finite support).
 Since $TS=ST$, arguing as in the proof of Lemma \ref{lem:commutant}, it is easy to show that $S_{m,n}$ and $S_{n,m}$
 have to satisfy
 \begin{align*}
  S_{n,m}S_{w^{(m)},w^{(m)},b}&=S_{w^{(n)},w^{(n)},c}S_{n,m}\\
  S_{m,n}S_{w^{(n)},w^{(n)},c}&=S_{w^{(m)},w^{(m)},b}S_{m,n}.
 \end{align*}
 Choosing $b=0$ and $c=\delta_0=(1,0,\dots)$ so that $S_{w,w,\delta_0}=\textrm{Id}$, we get $S_{m,n}=S_{n,m}=0$.

 Choose now $T\in\{C_\varphi\}'$ such that $T_{k,l}=0$ for all $k,l\in I$ except for $T_{n,m}=S_{w^{(n)},w^{(m)},\delta_0}$.
 This last operator is a diagonal operator with diagonal entries equal to
 $$\frac{w_0^{(n)}}{w_0^{(m)}}\times\cdots\times\frac{w_{k-1}^{(n)}}{w_{k-1}^{m}}.$$
Since $|w_{k}^{(n)}|\leq |w_k^{(m)}|$ for all $k\in\NN$, these diagonal entries are bounded and thus
 $S_{w^{(n)},w^{(m)},\delta_0}$ is bounded. Since $S$ commutes with $C_\varphi$, we already know that there exist $a,d\in\CC^\NN$ such that
 $S_{m,m}=S_{w^{(m)},w^{(m)},a}$ and $S_{n,n}=S_{w^{(n)},w^{(n)},d}$. Now, since $S$ commutes with $T$, we also have
 \begin{equation}\label{eq:doublecommutant}
S_{w^{(n)},w^{(n)},d}S_{w^{(n)},w^{(m)},\delta_0}=S_{w^{(n)},w^{(m)},\delta_0}S_{w^{(m)},w^{(m)},a}.
 \end{equation}
 Observe that the first column of $S_{w^{(n)},w^{(n)},d}S_{w^{(n)},w^{(m)},\delta_0}$ is
 $$\begin{pmatrix}
    d_0\\ d_1 w_0^{(n)} \\ d_2 w_0^{(n)}w_1^{(n)}\\ \cdots
   \end{pmatrix}$$
   whereas the first column of $S_{w^{(n)},w^{(m)},\delta_0}S_{w^{(m)},w^{(m)},a}$ is
 $$\begin{pmatrix}
    a_0\\ a_1 w_0^{(n)} \\ a_2 w_0^{(n)}w_1^{(n)}\\ \cdots
   \end{pmatrix}.$$
Hence, $a=d$ and $S=\bigoplus_{m\in I}S_{w^{(m)},w^{(m)},a}$. The converse implication is immediate.
\end{proof}

We are now ready to prove Theorem \ref{thm:doublecommutantlarge}.
\begin{proof}[Proof of Theorem \ref{thm:doublecommutantlarge}]
 The proof follows the argument of \cite[Theorem 1]{SW}. Let $S\in\{C_\varphi\}''$ and $a\in\CC^\NN$ be such that $S=\bigoplus_{m\in I} S_{w^{(m)},w^{(m)},a}.$
 Observe that each vector $e_k$ of the canonical basis of $H_m$,
 namely $e_k=m^{-c_0^k s}$,
 $$S_{w^{(m)},w^{(m)},a}e_k=\sum_{i=0}^{+\infty}a_i S_{w^{(m)}}^i e_k.$$

 Let, for $l\geq 1,$ $Q_l(X)=\sum_{i=0}^l a_i X^i$ and
$$P_l=\frac{Q_0+\cdots+Q_l}{l+1}.$$
As explained in \cite{SW}, for each $m\in I,$
$$\|P_l(S_{w^{(m)}})\|\leq \|S_{w^{(m)}}\|$$
so that
$$\|P_l(C_\varphi)\|=\sup_{m\in I} \|P_l(S_{w^{(m)}})\|\leq \sup_{m\in I}\|S_{w^{(m)}}\|\leq \|C_\varphi\|.$$
Moreover, for each $m\in I$ and each $k\in\NN,$
\begin{align*}
 \|P_l(C_\varphi)(m^{-c_0^k s})-S(m^{-c_0^k s})\|&=\|P_l(S_{w^{(m)}}(m^{-c_0^k s}))-S(m^{-c_0^k s})\|\\
 &\xrightarrow{l\to+\infty}0
\end{align*}
since $Q_l(m^{-c_0^k s})$ goes to $S(m^{-c_0^k s})$.
Thus $(P_l(C_\varphi)(f))$ converges to $S(f)$ on a dense set.
Since $(\|P_l(C_\varphi)\|)$ is bounded,
we deduce that $(P_l(C_\varphi))$ converges to $S$ in the strong operator topology.
Hence $C_\varphi$ has the double commutant property.
\end{proof}

\begin{remark}
 In Section \ref{sec:c02}, we have shown that, for any $c\geq 1,$ setting $\varphi(s)=cs+(c-1)c_1/(c_0-1)$, 
 $C_{\varphi_c}$ belongs to $\{C_\varphi\}'.$ Moreover it is easy to check that for all $c,d\geq 1,$
 $C_{\varphi_c}$ and $C_{\varphi_d}$ commute. Since $C_\varphi$ has the double commutant property
 and $C_{\varphi_c}\notin \overline{\mathrm{alg}(C_\varphi)}^\sigma$ for some $c\geq 1,$ we know that $\{C_\varphi\}'$ contains
 operators which are not equal to some $C_{\varphi_c}$.
\end{remark}

\subsection{Application to sums of weighted shifts}

The method used in the previous subsection allows us to characterize when a direct sum of weighted shifts has the double commutant property. Let us fix $I\subset\NN$, $I\neq\varnothing$ and let us consider, for each $n\in\mathbb N$, $w^{(n)}$ a weight sequence and $H_n$ a separable infinite dimensional Banach space. If we assume that
$\sum_{n\in I}\|w^{(n)}\|_\infty<+\infty,$ one may define the direct sum $W=\bigoplus_{n\in I}S_{w^{(n)}}$ which is a bounded operator acting on $H=\bigoplus_{n\in I}H_n$.

We define an equivalence relation on $I$ by, for $m,n\in I,$ $m\sim n$ if and only if there exists $l\in\mathbb N^*,$ $m_0,\dots,m_l\in I$ such that $m=m_0,$ $n=m_l$ and,
for all $i\in\{1,\dots,l\}$,
$$\left(\frac{w_0^{(m_i)}\cdots w_k^{(m_i)}}{w_0^{(m_{i-1})}\cdots w_k^{(m_{i-1})}}\right)_k$$
is either bounded or bounded away from $0$.

\begin{theorem}
Under the above assumptions, $W$ has the double commutant property if and only if there is a single equivalence class for $\sim$.
\end{theorem}
\begin{proof}
Lemma \ref{lem:commutant} is still valid under these assumptions: $T\in\{W\}'$ if and only
if for each $(m,n)\in I^2$, there exists a sequence $a^{(m,n)}\in\CC^\NN$ such that, for each $(m,n)\in I^2,$
$$T_{m,n}=S_{w^{(m)},w^{(n)},a^{(m,n)}}.$$
Looking at the proof of Theorem \ref{thm:doublecommutantlarge}, we get that any $S\in\{W\}''$
may be written $S=\bigoplus_{m\in I}S_{w^{(m)},w^{(m)},a^{(m)}}$, $a^{(m)}\in\CC^\NN$ for each $m\in I$ and that, for $m,n\in I,$ $m\neq n$, if the sequence
$$\left(\frac{w_0^{(m)}\cdots w_k^{(m)}}{w_0^{(n)}\cdots w_k^{(n)}}\right)_k$$
is either bounded or bounded away from $0$, then $a^{(m)}=a^{(n)}$ (the proof of Theorem \ref{thm:doublecommutantlarge} works immediately if this sequence is bounded, if it is bounded away from $0$, then we consider $T\in\{W\}'$ with $T_{k,l}=0$ except $T_{m,n}=S_{w^{(m)},w^{(n)},\delta_0}$ and we repeat the proof). Hence, if there is only one equivalence class for $\sim$, we get that $a^{(m)}=a^{(n)}$ for all $m,n\in I$ and we conclude as in the proof of Theorem \ref{thm:doublecommutantlarge} that $W$ has the double commutant property.

Suppose on the contrary that there are at least two equivalence classes. Let $(I_j)_{j\in J}$ be the different equivalent classses with $\textrm{card}(J)\geq 2$. The key point is to observe that the commutant of $W$ is smaller than in the previous case. Indeed, for any $T\in\{W\}'$, if $m\in I_j$ and $n\in I_{j'}$ for $j\neq j',$ then $T_{m,n}=0$ since for no $a\neq 0$, the operator $S_{w^{(m)},w^{(n)},a}$ is bounded: if $a_p\neq 0,$ then
\begin{align*}
\left|\langle S_{w^{(m)},w^{(n)},a}e_{p+k},e_k\rangle\right|&=\left|a_p \frac{w_0^{(m)}\cdots w_{p+k-1}^{(m)}}{w_0^{(n)}\cdots w_{k-1}^{(n)}}\right|\\
&\geq |a_p|\cdot \left|\frac{w_0^{(m)}\cdots w_{p+k-1}^{(m)}}{w_0^{(n)}\cdots w_{p+k-1}^{(n)}}\right|\times \frac 1{\|w^{(n)}\|_\infty^p}
\end{align*}
is unbounded.

The commutant being smaller, the double commutant becomes larger. To observe this, let us fix $j_0\in J$ and we set $a^{(m)}=\delta_0$ provided $m\in I_{j_0}$, $a^{(m)}=2\delta_0$ provided $m\in I\backslash I_{j_0}$. Then the operator
$S=\bigoplus_{m\in I}S_{w^{(m)},w^{(m)},a^{(m)}}$ belongs to $\{W\}''$ since
we have no equation like \eqref{eq:doublecommutant} to imply that $a^{(m)}$ should be equal to $a^{(n)}$ if $m\in I_{j_0}$ and $n\in I\backslash I_{j_0}$. 
We claim that this operator does not belong to $\overline{\textrm{alg}(W)}^\sigma$. 
Indeed, consider $(e_{m,n})_{m\in I,n\in \NN}$ the orthonormal basis associated to $H=\bigoplus_{m\in I}H_m.$ For any $P\in\mathbb C[X]$, $P(X)=\sum_{k=0}^d a_k X^k,$ we have
$$\langle P(W)(e_{m,0}),e_{m,0}\rangle=a_0\textrm{ for all }m\in I$$
whereas
$$\langle S(e_{m,0}),e_{m,0}\rangle=
\left\{
\begin{array}{ll}
1&\textrm{ if }m\in I_{j_0}\\
2&\textrm{ otherwise.}
\end{array}\right.$$
\end{proof}
\begin{example}
Let us define $w$ and $w'$ by $w_0=1$,
$$w_k=\left\{
\begin{array}{ll}
2&\textrm{ if }k\in\bigcup_{p\in\NN}[2^{2p},2^{2p+1})\\
1/2&\textrm{ if }k\in\bigcup_{p\in\NN}[2^{2p+1},2^{2p+2}),
\end{array}\right.,$$
and $w'_k=\frac{1}{w_k}$ for all $k\geq 0.$
Then the sequence $(\frac{w_0\cdots w_k}{w'_0\cdots w'_k})$ is neither bounded nor bounded away from $0$. Let $w^{(n)}=w$ if $n$ is odd and $w^{(n)}=w'$ is $n$ is even. Then $W=\bigoplus_{n\in \NN}S_{w^{(n)}}$ fails to have the double commutant property.
\end{example}

\medskip

\noindent\textbf{Acknowledgment.} 
X. Yao was supported by National Science Foundation of China (Nos. 12171373, 11701434).

\end{document}